\documentclass[english,a4paper,11pt,twoside]{article}
\usepackage{amsmath,amsthm,amssymb,amsfonts}
\usepackage{enumerate,mathrsfs}
\usepackage{hyperref}
\usepackage[cmtip,all]{xy}

\newtheoremstyle{theoremstyle}
  {10pt}      
  {5pt}       
  {\itshape}  
  {}          
  {\bfseries} 
  {:}         
  {.5em}      
  {}          

\newtheoremstyle{examplestyle}
  {10pt}      
  {5pt}       
  {}          
  {}          
  {\bfseries} 
  {:}         
  {.5em}      
  {}          

\theoremstyle{theoremstyle}
\newtheorem{theorem}{Theorem}[section]
\newtheorem*{theorem*}{Theorem}
\newtheorem{lemma}[theorem]{Lemma}
\newtheorem{proposition}[theorem]{Proposition}
\newtheorem*{proposition*}{Proposition}
\newtheorem{corollary}[theorem]{Corollary}
\newtheorem*{corollary*}{Corollary}

\theoremstyle{definition}
\newtheorem{example}[theorem]{Example}
\newtheorem{definition}[theorem]{Definition}
\newtheorem{definition*}{Definition}
\newtheorem{remark}[theorem]{Remark}
\newtheorem{remark*}{Remark}

\newcommand{\too}{{\longrightarrow}}

\newcommand{\Z}{{\mathbb{Z}}}

\newcommand{\R}{{\mathbb{R}}}
\newcommand{\N}{{\mathbb{N}}}

\newcommand{\cat}{\mathscr}
\newcommand{\inj}{\hookrightarrow}
\newcommand{\surj}{\twoheadrightarrow}
\newcommand{\iso}{\stackrel{\cong}{\longrightarrow}}
\newcommand{\cof}{\rightarrowtail}

\newcommand{\g}{{C_2}}

\DeclareMathOperator*{\colim}{colim}
\DeclareMathOperator*{\hocolim}{hocolim}

\begin{document}
\author{Kristian Jonsson Moi\footnote{Supported by the Danish National Research Foundation through the Centre for Symmetry and Deformation (DNRF92) and partially supported by ERC adv. grant No.228082.}}
\title{Equivariant loops on classifying spaces}
\maketitle
\begin{abstract}We compute the homology of the space of equivariant loops on the classifying space of a simplicial monoid $M$ with anti-involution, provided $\pi_0 (M)$ is central in the homology ring of $M$. The proof is similar to McDuff and Segal's proof of the group completion theorem.

Then we compute the homology of the $C_2$-fixed points of a Segal-type model of the algebraic $K$-theory of an additive category with duality. As an application we show that this fixed point space is sometimes group complete, but not in general.
  
\end{abstract}
\tableofcontents
\newpage
\section{Introduction} It is well known that for a group-like simplicial monoid $M$ the natural map
\[\eta_M \colon |M| \to \Omega |BM| \]
is a weak homotopy equivalence. In the non-group-like case the classical group completion theorem \cite{mcduff-segal}, \cite[Q.4]{filtr} states that for a simplicial monoid $M$ satisfying certain conditions $\eta_M$ induces an isomorphism of $H_\ast(M)$-modules
\[H_\ast(M)[\pi_0(M)^{-1}] \iso H_\ast(\Omega |BM|).\]
The first task of this paper is to investigate the corresponding situation when $M$ has an anti-involution, i.e. a map $\alpha \colon M \to M$ such that $\alpha(m n) =  \alpha( n)\alpha(m )$. This extra structure allows us to define maps $w_i \colon B_iM \to B_iM$ given by 
\[w_p(m_1,m_2, \ldots , m_{p}) = (\overline{m_p},   \ldots,\overline{m_2}, \overline{m_{1}}) \]
which satisfy $w_n \circ w_n = id$ and compatibility relations with the simplicial structure maps (see \ref{rel:realsimp}). As a consequence the realisation $|BM|$ is a $\g$-space in a natural way.

We let $\R^{1,1}$ denote the minus-representation of $\g$ on $\R$ and write $S^{1,1}$ for its one-point compactification. Another model for this $\g$-space is given by taking the unit interval $[0,1]$ with the action $t \mapsto 1-t$ and collapsing the boundary. For a $\g$-space $X$ we write $\Omega^{1,1}X$ for the space $Map_\ast(S^{1,1},X)$ with the conjugation action of $\g$. In his thesis \cite{nisan} Nisan Stiennon has shown that $\eta_M$ is in fact an equivariant map
\[|M| \to \Omega^{1,1}|BM|\]
and that if $M$ is group-like then the induced map on fixed points
\[|M|^\g \to (\Omega^{1,1}|BM|)^\g\]
is a weak equivalence. In view of the group completion theorem it is then natural to ask what happens when $M$ is not group-like. The answer is as follows (see \ref{theorem:mainmon}):
\begin{theorem*}
  Let $M$ be a simplicial monoid with anti-involution such that $\pi_0 M$ is in the center of $H_\ast(M)$. Then the map
\[\eta^\g_M \colon |M|^\g \to (\Omega^{1,1}|B^{1,1}M|)^\g\]
induces an isomorphism 
\[\pi_0(M^\g)[\pi_0(M)^{-1}] \iso \pi_0(\Omega^{1,1}|B^{1,1}M|)^\g\]
of $\pi_0(M)$-sets and an isomorphism of $H_\ast(M)$-modules 
\[H_\ast(M^\g)[\pi_0(M)^{-1}] \iso H_\ast((\Omega^{1,1}|B^{1,1}M|)^\g).\]
\end{theorem*}

In \cite[§4]{catcoh} Segal proved a variant of the group completion theorem for $\Gamma$-spaces. Shimakawa  \cite{shimakawa} later considerer the $G$-equivariant situation for $G$ a finite group. He gave a way to deloop $\Gamma_G$-spaces and proved a group completion statement for these deloopings provided one is delooping with respect to a representation sphere $S^W$ such that $W^G \neq 0$. In this paper we consider the case $G = \g$ and $W = \R^{1,1}$. Since $(\R^{1,1})^\g = 0$, Shimakawa's result does not apply. We describe a construction of a Segal-type delooping of an additive category with duality with respect to $S^{1,1}$ and prove theorem \ref{theorem:maincat} which is analogous to \ref{theorem:mainmon}. As an application we make $\pi_0$-computations for symmetric (\ref{prop:symm(Z)}) and symplectic (\ref{prop:symp(Z)}) form spaces over $\Z$. 

The author would like to thank S{\o}ren Galatius and Lars Hesselholt for suggesting the problem for monoids and additive categories with duality, respectively. He would also like to thank Emanuele Dotto for important input, and Ib Madsen for guidance along the way. 

\section{Homology fibrations}
In this section we collect some basic facts about homology fibrations of simplicial, and bisimplicial sets. We make no claim to originality; the results here can either be found in \cite{grcomprev},\cite[IV,5]{gj} or \cite{jardine} or are easy consequences of the results there. 

A map of spaces or simplicial sets inducing an isomorphism on integral homology will be called a \emph{homology equivalence}.
\begin{definition}
  A commuting square
\[\xymatrix{ A \ar[r] \ar[d] & B \ar[d]^{f} \\
C\ar[r]_g & D }\]
 of simplicial sets is called homology cartesian if when we factor $f \colon B \to D$ as a trivial cofibration followed by a fibration 
\[B \stackrel{\simeq}{\cof} W \surj D\]
the induced map from $A$ to the pullback $C \times_D W$ is a homology equivalence.
\end{definition}
Note that a homotopy cartesian square is automatically homology cartesian. Just as for homotopy cartesian squares it doesn't matter which factorization we use or whether we choose to factor $f$ or $g$. By analogy with the case of homotopy cartesian squares \cite[II,8.22]{gj} we have:
\begin{lemma}\label{lemma:I+II}
  Let 
\[\xymatrix{A \ar[r] \ar[d] \ar@{} [dr] |{{\bf I}}& A' \ar[r] \ar[d]\ar@{} [dr] |{{\bf II}} & A'' \ar[d]\\
B \ar[r] & B' \ar[r]  & B''}\]
be a diagram of simplicial sets such that the square {\bf II} is homotopy cartesian. Then {\bf I} is homology cartesian if and only if the outer rectangle ${\bf I} + {\bf II}$ is homology cartesian.
\end{lemma}
The proof is an exercise in the fact that pulling back along fibrations of simplicial sets preserves weak equivalences.

Let $Y$ be a simplicial set. An $m$-simplex $\sigma \in Y_m$ of $Y$ corresponds to a unique map $\Delta^m \to Y$ which we will also call $\sigma$. The simplices of $Y$ form a category $Simp(Y)$ where an object is a map $\sigma \colon \Delta^m \to Y$, and where a morphism 
\[(\sigma \colon \Delta^m \to Y) \to (\tau \colon \Delta^n \to Y)\]
is a map $\alpha \colon [m] \to [n]$ in $\Delta$ such that the diagram
\[\xymatrix{\Delta^{m} \ar[rd]_\sigma \ar[rr]^{\alpha_\ast} & & \Delta^n \ar[dl]^\tau \\
& Y & }\]
commutes. Composition is given by composition of maps in $\Delta$. 

Let $f \colon X \to Y$ be a map of simplicial sets. Then, for any simplex $\sigma \colon \Delta^m \to Y$ we define $f^{-1}(\sigma)$ to the be the pullback in the square   
\[\xymatrix{f^{-1}(\sigma) \ar[r] \ar[d] & X \ar[d]^f \\
\Delta^m \ar[r]_\sigma \ar[r] & Y.}\]
For a diagram 
\[\xymatrix{\Delta^{m} \ar[rd]_\sigma \ar[rr]^{\alpha_\ast} & & \Delta^n \ar[dl]^\tau \\
& Y & }\]
there is an induced map 
\[f^{-1}(\alpha_\ast) \colon f^{-1}(\sigma) \to f^{-1}(\tau).\]
The assignments 
\[\sigma \mapsto f^{-1}(\sigma)\] 
and 
\[\left(\alpha_\ast \colon \left(\sigma \colon \to \Delta^m\right) \to \left(\tau \colon \to \Delta^n\right) \right) \mapsto \left(f^{-1}\left(\alpha_\ast\right) \colon f^{-1}\left(\sigma\right) \to f^{-1}\left(\tau\right)\right)\]
form the object and morphism components, respectively, of a functor 
\[f^{-1} \colon Simp(Y) \to sSet.\]
If $g \colon Z \to Y$ is another map to $Y$, then a map $h \colon X \to Z$ of objects over $Y$ induces a natural transformation
\[h_\ast \colon f^{-1} \to g^{-1}.\]
It is worth noting that $\colim f^{-1} \cong X$ as objects over $Y$, in particular $\colim id_Y^{-1} \cong Y$. The homotopy colimit $\hocolim f^{-1}$ is the diagonal of the bisimplicial set $\coprod_\ast f^{-1}$ (see \cite[IV,1.8]{gj}) which is given in degree $n$ by
\[(\coprod {}_\ast f^{-1})_n = \coprod_{\sigma \in N_n Simp(Y)} f^{-1}(\sigma(0)).\]
For each $n$ there is a map of simplicial sets
\[\coprod_{\sigma \in N_n Simp(Y)} f^{-1}(\sigma(0)) \to X.\]
Considering $X$ as a bisimplicial set which is constant in the nerve direction these maps assemble to a map of bisimplicial sets 
\[\gamma_f \colon  \coprod {}_\ast f^{-1} \to X.\]
\begin{lemma}
  The diagonal 
\[d\gamma_f \colon \hocolim f^{-1} \to X\]
is a weak equivalence.
\end{lemma}
For a proof see \cite[IV,5.1]{gj}.
\begin{lemma}\label{lemma:hfibprop}
  Let $f \colon X \to Y$ be a map of simplicial sets. The following are equivalent:
  \begin{enumerate}
  \item For every simplex $\sigma \colon \Delta^m \to Y$ the pullback diagram
\[\xymatrix{f^{-1}(\sigma) \ar[r] \ar[d] & X \ar[d]^f \\
\Delta^m \ar[r]_\sigma \ar[r] & Y}\]
is homology cartesian.
\item For any pair of simplices $\sigma \colon \Delta^m \to Y$ and $\tau \colon \Delta^n \to Y$ and for any diagram
\[\xymatrix{\Delta^{m} \ar[rd]_\sigma \ar[rr]^{\alpha_\ast} & & \Delta^n \ar[dl]^\tau \\
& Y & }\]
the induced map on pullbacks along $f$
\[f^{-1}(\alpha_\ast) \colon f^{-1}(\sigma) \to f^{-1}(\tau)\]
is a homology equivalence.
  \end{enumerate}
\end{lemma}
In the proof of this lemma we will use the following result which is proven in \cite[IV,5.11]{gj}, see \cite{grcomprev} for a different proof of \ref{lemma:hfibprop}.
\begin{theorem}\label{thm:homolB}
  Let $X \colon I \to sSet$ be a functor such that for any morphism $i\to j$ in $I$ the induced map $X(i) \to X(j)$ is a homology equivalence, then for all objects $i$ of $I$ the pullback diagram
\[\xymatrix{X(i) \ar[r] \ar[d] & \hocolim_I X \ar[d] \\
\ast \ar[r] & NI }\]
is homology cartesian.
\end{theorem}
\begin{proof}[Proof of lemma \ref{lemma:hfibprop}]
  $\mathit{1} \implies \mathit{2}$: We begin by factoring $f$ as 
\[X \stackrel{g}{\cof} W \stackrel{\bar{f}}{\surj} Y,\]
where $g$ is a weak equivalence. Condition 1 says precisely that the natural transformation $g_\ast \colon f^{-1} \to \bar{f}^{-1}$ has components which are homology equivalences. Since $\bar{f}$ is a fibration the functor $\bar{f}^{-1}$ sends all maps in $Simp(Y)$ to weak equivalences. Therefore, a map $\alpha \colon \sigma \to \tau$ in $Simp(Y)$ gives a naturality square
\[\xymatrix{f^{-1}(\sigma) \ar[r]^{f^{-1}(\alpha_\ast)} \ar[d]_{g_{\ast,\sigma}} & f^{-1}(\tau) \ar[d]^{g_{\ast,\tau}} \\
\bar{f}^{-1}(\sigma) \ar[r]^{\simeq}_-{\bar{f}^{-1}(\alpha_\ast)} & \bar{f}^{-1}(\tau) }\]
where the vertical maps are homology equivalences and the lower horizontal map is a weak equivalence. It follows that $f^{-1}(\alpha_\ast)$ is a homology equivalence. \\
$\mathit{2} \implies \mathit{1}$ (cp. \cite[IV, 5.18]{gj}):
For every simplex $\sigma \colon \Delta^m \to Y$ there is a diagram of bisimplicial sets
\[\xymatrix{f^{-1}(\sigma) \ar@{} [dr] |{{\bf I}} \ar[r] \ar[d] & \coprod {}_\ast f^{-1}  \ar[r]^-\simeq \ar[d] \ar@{} [dr] |{{\bf II}}& X \ar[d]^f\\ \Delta^m \ar[r] \ar[d]_\simeq \ar@{} [dr] |{{\bf III}}& \coprod {}_\ast id_Y^{-1} \ar[r]^-\simeq \ar[d]^\simeq & Y \\
\ast \ar[r] & \coprod_{NSimp(Y)} \ast. &}\]
Write $d({\bf I})$ for the square obtained by taking diagonals in the square ${\bf I}$ and similarly for the other sub-diagrams. The square $d({\bf I} + {\bf III})$ is
\[\xymatrix{f^{-1}(\sigma) \ar[r] \ar[d] & \hocolim f^{-1} \ar[d] \\
\ast \ar[r]_-\sigma & NSimp(Y), }\]
which is homology cartesian by \ref{thm:homolB}. Since the square $d({\bf III})$ is homotopy cartesian it follows by \ref{lemma:I+II} that $d({\bf I})$ is homology cartesian. The square $d({\bf II})$ is also homotopy cartesian so it follows, again by \ref{lemma:I+II}, that $d({\bf I} + {\bf II})$ is homology cartesian.
\end{proof}

\begin{definition}
  A map $f \colon X \to Y$ of simplicial sets is called a homology fibration if it satisfies one (and hence both) of the conditions of lemma \ref{lemma:hfibprop}.
\end{definition}

\begin{definition}
  A map $p \colon E \to B$ of topological spaces is called a homology fibration if for any point $b \in B$ the natural map from the fiber $F_b$ at $b$ to the homotopy fiber $hF_b$ induces an isomorphism on integral homology. 
\end{definition}

The relation between the two kinds of homology fibrations is given as follows:
\begin{theorem}\cite[4.4]{grcomprev}\label{hfibeq}
  A map $f \colon X \to Y$ of simplicial sets is a homology fibration if and only if the induced map on realizations $|f| \colon |X| \to |Y|$ is a homology fibration of topological spaces.
\end{theorem}

Recall from \cite{confiter} the Segal edgewise subdivision functor $Sd \colon sSet \to sSet$. An important property of this construction is that the realization of a simplicial set $X$ is naturally homeomorphic to the realization of its subdivision $SdX$. Knowing this, we get the next lemma from \ref{hfibeq}.

\begin{lemma}\label{lemma:hfibsd}
  A map $f \colon X \to Y$ of simplicial sets is a homology fibration if and only if the induced map $Sdf \colon SdX \to SdY$ is a homology fibration.
\end{lemma}

The next lemma follows easily from condition 2 of lemma \ref{lemma:hfibprop}.

\begin{lemma}\label{lemma:basechange}
  Homology fibrations are closed under base change.
\end{lemma}

\begin{lemma}
   Let $f \colon X  \to Y$ be a homology fibration and let $g \colon Z \to Y$ be any map. Then the pullback square
\[\xymatrix{Z \times_Y X \ar[r] \ar[d]_h & X \ar[d]^f \\
Z \ar[r]_g & Y}\]
is homology cartesian.
\end{lemma}

\begin{proof}
 Factor the map $f$ as 
\[X \stackrel{i}{\cof} W \stackrel{\bar{f}}{\surj} Y,\]
where $i$ is also a weak equivalence. There is an induced factorization  
\[Z \times_Y X \stackrel{j}{\to} Z \times_Y W \stackrel{\bar{h}}{\surj} Z\]
of $h$ and we must show that $j$ is a homology equivalence. 

Let 
\[H_q( h^{-1},\Z) \colon Simp(Z) \to Ab\]
be the composite functor given by 
\[\sigma \mapsto h^{-1}(\sigma)  \mapsto H_q( h^{-1}(\sigma), \Z),\]
where $Ab$ is the usual category of abelian groups. The natural transformation $j_\ast \colon H_q( h^{-1},\Z) \to H_q( \bar{h}^{-1},\Z)$ is an isomorphism of functors, since by lemma \ref{lemma:basechange} $h$ is a homology fibration. Recall that for a functor $F \colon I \to Ab$ the \emph{translation object} $EF$ of $F$ is the simplicial abelian group given in degree $n$ by
\[EF_n = \bigoplus_{i_0 \to \cdots \to i_n} F(i_0) \]
with structure maps as for $\coprod_\ast F$ (see \cite[IV,2.1]{gj}). The map 
\[E H_q( h^{-1},\Z) \to E H_q( \bar{h}^{-1},\Z)\]
induced by $j_\ast$ is an isomorphism. By \cite[IV,5.1]{gj} there is a first quadrant spectral sequence
\[E^{p,q}_2 = \pi_p E  H_q( h^{-1},\Z) \implies H_{p+q}(Z \times_Y X,\Z) \]
and a corresponding one for $\bar{h}^{-1}$ converging to $H_{p+q}(Z \times_Y W,\Z)$. The map $j$ induces an isomorphism of $E_2$-pages and is therefore a homology equivalence by the comparison theorem for spectral sequences.
\end{proof}

For a homology fibration $f \colon X \to Y$ the functor $H_q( f^{-1},\Z)$ sends all maps to isomorphism and hence factors through the groupoid $GSimp(Y)$ obtained from $Simp(Y)$ by inverting all morphisms (see \cite[p. 235]{gj}). This groupoid is naturally equivalent to the fundamental groupoid of the realization $|X|$ (see \cite[III,1.1]{gj}). If for any pair of maps $\xi,\zeta \colon \sigma \to \tau$ in $GSimp(Y)$ the induced maps 
\[\xi_\ast,\zeta_\ast \colon H_q( f^{-1}(\sigma),\Z) \to H_q( f^{-1}(\tau),\Z)\]
agree, we say that the fundamental groupoid acts trivially on the homology of the fibers of $f$. If $Y$ is connected then for any simplex $\rho$ in $Y$ there is a unique isomorphism of functors 
\[(\sigma \mapsto H_q( f^{-1}\sigma,\Z)) \cong (\sigma \mapsto H_q( f^{-1}\rho,\Z))\]
whose value at $\rho$ is the identity map.
\begin{proposition}\label{lemma:lproper}
  Let $f \colon X \to Y$ be a homology fibration such that the fundamental groupoid of $Y$ acts trivially on the homology of the fibers of $f$. Then, for any homology equivalence $g \colon Z \to Y$ the induced map
\[g' \colon Z \times_Y X \to X \]
is a homology equivalence. 
\end{proposition}
\begin{proof}

Assume, without loss of generality, that $Y$ is connected and choose a fiber $F$ over some vertex of $Y$. Then, by \cite[IV,5.1]{gj}, there are Serre spectral sequences, for $f$
\[E_2^{p,q}=H_p(X,H_q(F)) \implies H_{p+q}(Y), \]
and for the pullback of $f$ along $g$
\[E_2^{p,q}=H_p(Z,H_q(F)) \implies H_{p+q}(Z \times_Y X). \]
The map of $E_2$-pages induced by $g'$ is an isomorphism by the universal coefficient theorem and the fact that $g$ is a homology equivalence. It follows that $g'$ is a homology equivalence.

\end{proof}

We now turn to bisimplicial sets. A map of bisimplicial sets will be called a homology equivalence if the induced map on diagonals is a homology equivalence.

\begin{lemma}\label{lemma:bisimphfibprop}
  Let $f \colon X \to Y$ be a map of bisimplicial sets. The following are equivalent:
  \begin{enumerate}
  \item The diagonal $df \colon dX \to dY$ is a homology fibration.
\item \label{lemma:bisimphfibprop:prop2}For any pair of bisimplices $\sigma \colon \Delta^{m,n} \to Y$ and $\tau \colon \Delta^{p,q} \to Y$ and for any diagram
\[\xymatrix{\Delta^{m,n} \ar[rd]_\sigma \ar[rr]^{(\alpha,\beta)_\ast} & & \Delta^{p,q} \ar[dl]^\tau \\
& Y & }\]
the induced map on pullbacks along $f$
\[f^{-1}(\alpha,\beta)_\ast \colon f^{-1}(\sigma) \to f^{-1}(\tau)\]
is a homology equivalence.
  \end{enumerate}
\end{lemma}

\begin{proof}
$1 \implies 2:$ Given a bisimplex $\sigma \colon \Delta^{m,n} \to Y$ we choose a vertex $v$ of $\sigma$. Since pullbacks commute with diagonals, we get a diagram
\[\xymatrix{ df^{-1}(v) \ar[r] \ar[d] &  df^{-1}(\sigma) \ar[r] \ar[d] & dX \ar[d]^{df} \\
\Delta^0 \ar[r] & \Delta^m \times \Delta^n \ar[r] & dY }\]
in which the two squares and the outer rectangle are pullback diagrams. The middle vertical map is a homology fibration, by \ref{lemma:basechange}, since it is a pullback of $df$. The lower left map is a weak equivalence, so it follows by \ref{lemma:lproper} that the induced map $df^{-1}(v) \to df^{-1}(\sigma)$ is a homology equivalence. By definition this means that the map $f^{-1}(v) \to f^{-1}(\sigma)$ is a homology equivalence. A map $(\alpha,\beta) \colon \sigma \to \tau$ gives a commuting triangle
\[\xymatrix{& f^{-1}(v) \ar[dl] \ar[dr] & \\
f^{-1}(\sigma) \ar[rr]_{f^{-1}(\alpha,\beta)} & & f^{-1}(\tau) .}\]
By the argument above the two downward maps are homology equivalences, so it follows that $f^{-1}(\alpha,\beta)$ is too.

$2 \implies 1:$ The proof follows roughly the same outline as the corresponding proof for simplicial sets. As for simplicial sets there is a category $Simp(Y)$ of bisimplices of $Y$ and condition 2 says that the functor $f^{-1} \colon Simp(Y) \to bisSet$ takes values in homology equivalences. Composing $f^{-1}$ with the diagonal functor $d \colon bisSet \to sSet$ gives a functor $df^{-1} \colon Simp(Y) \to sSet$ taking values in homology equivalences. For a simplex $\sigma \colon \Delta^n \to dY$ there is a diagram of bisimplicial sets \cite[IV,5.18]{gj}
\[\xymatrix{df^{-1}(\sigma) \ar@{} [dr] |{{\bf I}} \ar[r] \ar[d] & \coprod {}_\ast (df)^{-1}  \ar[r]^-\simeq \ar[d] \ar@{} [dr] |{{\bf II}}& dX \ar[d]^{df}\\ \Delta^n \ar[r] \ar[d]_\simeq \ar@{} [dr] |{{\bf III}}& \coprod {}_\ast id_{dY}^{-1} \ar[r]^-\simeq \ar[d]^\simeq & dY \\
\ast \ar[r] & \coprod_{NSimp(Y)} \ast. &}\]
The two horizontal maps in {\bf II} are weak equivalences by \cite[IV,5.17]{gj}.
We now conclude as in the proof of \ref{lemma:hfibprop}, that the rectangle ${\bf I}+{\bf II}$ is homology cartesian.
\end{proof}

\begin{definition}
  A map $f \colon X \to Y$ of bisimplicial sets is called a homology fibration if it satisfies one (and hence both) of the conditions of lemma \ref{lemma:bisimphfibprop}.
\end{definition}

The exposition of propositions \ref{prop:leveltoglobal} and \ref{cor:leveltoglobal}, and their proofs follows \cite{jardine} closely but any errors or omissions are my own. If $X$ is a bisimplicial set we will write $X_n$ for the simplicial set 
\[[p] \mapsto X_{n,p}.\]
\begin{proposition}\label{prop:leveltoglobal}
  Let $f \colon X \to Y$ be a map of bisimplicial sets such that for each $n \geq 0$ the map $f_n \colon X_n \to Y_n$ is a Kan fibration. Assume that for each $\theta \colon [m] \to [n]$ and each $v \in Y_{n,0}$ the induced map on fibers $f_n^{-1}(v) \to f_m^{-1}(\theta^\ast(v))$ is a homology equivalence. Then $f$ is a homology fibration.
\end{proposition}

\begin{proof}
  We show that $f$ satisfies condition \ref{lemma:bisimphfibprop:prop2} of lemma \ref{lemma:bisimphfibprop}. Given a bisimplex $\tau \colon \Delta^{p,q} \to Y$ choose a vertex $v$ of $\Delta^q$ and let $(id_{[p]},v)_\ast \colon \Delta^{p,0} \to \Delta^{p,q}$ be the corresponding map of bisimplicial sets. In level $n$ we can form the iterated pullback
\[\xymatrix{\coprod_{\gamma \in \Delta^p_n} f_n^{-1}(v) \ar[r]^-{v_\ast} \ar[d] & \coprod_{\gamma \in \Delta^p_n} f_n^{-1}((\gamma,id_{[q]})^\ast\tau_n) \ar[r] \ar[d] & X_n \ar@{>>}[d]^{f_n} \\
\coprod_{\gamma \in \Delta^p_n} \Delta^0 \ar[r]^\simeq_{\coprod v} & \coprod_{\gamma \in \Delta^p_n} \Delta^q \ar[r]_-{\coprod (\gamma,id_{[q]})^\ast\tau_n} & Y_n }\]
where the map $v_\ast$ is a weak equivalence since $f_n$ is a fibration. This says that the map $f^{-1}((id_{[p]},v)^\ast\tau) \to  f^{-1}(\tau)$ is a levelwise homology equivalence, and hence a homology equivalence by \cite[IV,2.6]{gj}. Therefore it suffices to consider diagrams of the form
\[\xymatrix{\Delta^{m,0} \ar[rd]_v \ar[rr]^{(\alpha,id)_\ast} & & \Delta^{p,0} \ar[dl]^w \\
& Y. & }\] 
In level $n$ the induced map $f^{-1}(v) \to  f^{-1}(w)$ fits as the top row in the diagram
\[\xymatrix{\coprod_{\gamma \in \Delta^m_n} f_n^{-1}((\gamma,id)^\ast v) \ar[r] & \coprod_{\delta \in \Delta^p_n} f_n^{-1}((\delta,id)^\ast w) \\
\coprod_{\gamma \in \Delta^m_n} f_m^{-1}( v) \ar[u] & \\
\coprod_{\gamma \in \Delta^m_n} f_p^{-1}(w) \ar[u] \ar[r] & \coprod_{\delta \in \Delta^p_n} f_p^{-1}(w) \ar[uu]},\]
The vertical maps are homology equivalences by assumption on $f$ and the lower horizontal map becomes the weak equivalence
\[\alpha_\ast \times id \colon \Delta^m \times f_p^{-1}(w) \to \Delta^p \times f_p^{-1}(w) \]
when we take diagonals.
\end{proof}

\begin{corollary}\label{cor:leveltoglobal}
   Let $f \colon X \to Y$ be a map of bisimplicial sets such that for each $n \geq 0 $ the map $f_{n} \colon X_{n} \to Y_{n}$ is a homology fibration and for each $v \in Y_{n,0}$ the induced map on fibers $f_n^{-1}(v) \to f_m^{-1}(\theta^\ast(v))$ is a homology equivalence. Then $f$ is a homology fibration.
\end{corollary}

\begin{proof}
  We begin by factoring the map $f \colon X \to Y$ as a levelwise trivial cofibration followed by a levelwise fibration $X \stackrel{g}{\too} W \stackrel{h}{\too} Y$. Given a bisimplex $\sigma \colon \Delta^{p,q} \to Y$ we get a diagram of bisimplicial sets
\[\xymatrix{f^{-1}(\sigma) \ar[r] \ar[d] & X \ar[d]^g \ar@/^2pc/[dd]^{f}\\
h^{-1}(\sigma)\ar[r] \ar[d] & W \ar[d]^h\\
\Delta^{p,q} \ar[r] & Y, }\]
which en level $n$ looks like
\[\xymatrix{\coprod_{\theta \in \Delta^p_n}f_n^{-1}((\theta,id_{[q]})^\ast\sigma_n) \ar[r] \ar[d] & X_n \ar@{>->}[d]_\simeq^{g_n} \ar@/^2pc/[dd]^{f_n}\\
\coprod_{\theta \in \Delta^p_n}h_n^{-1}((\theta,id_{[q]})^\ast\sigma_n)\ar[r] \ar[d] & W_n \ar@{>>}[d]^{h_n}\\
\coprod_{\theta \in \Delta^p_n}\Delta^{q} \ar[r]_-{\coprod (\theta,id_{[q]})^\ast\sigma_n}& Yn .}\]
Since $f_n$ is a homology fibration the upper left vertical map induces a homology equivalence on each summand, and is therefore a homology equivalence. This says that the map $f^{-1}(\sigma) \to h^{-1}(\sigma)$ is a levelwise homology equivalence and hence it is a homology equivalence by \cite[IV,2.6]{gj}. 

Given a vertex $v \in Y_{n,0}$ and a map $\theta \colon [m] \to [n]$ there is a commuting square of fibers
\[\xymatrix{f_n^{-1}(v) \ar[r] \ar[d] & f_m^{-1}(\theta^\ast v) \ar[d] \\
h_n^{-1}(v) \ar[r]  & h_m^{-1}(\theta^\ast v) }\]
The vertical maps are homology equivalences since $f$ is a levelwise homology fibration and the upper horizontal map is a homology equivalence by assumption on $f$. From this we see that lower horizontal map is a homology equivalence which implies that $h$ satisfies the conditions of \ref{prop:leveltoglobal}. A map $\sigma \to \tau$ in $Simp(Y)$ induces a square of pullbacks
\[\xymatrix{f^{-1}(\sigma) \ar[r] \ar[d] & f^{-1}(\tau) \ar[d]\\
 h^{-1}(\sigma) \ar[r]& h^{-1}(\tau)}.\]
The vertical maps are homology equivalences by the arguments above and the lower horizontal map is homology equivalence since $h$ is a homology fibration. It follows that the top horizontal map is a homology equivalence so that $f$ is a homology fibration.
\end{proof}

For a bisimplicial set $X$ we write $Sd_hX$ for the Segal edgewise subdivision of $X$ in the first (horizontal) variable and $Sd_vX$ for the subdivision in the second (vertical) variable. Clearly $Sd_hSd_vX = Sd_vSd_hX$ and $dSd_hSd_vX = Sd(dX)$. 

\begin{lemma}\label{lemma:sdbihfib}
  Let $f \colon X \to Y$ be a map of bisimplicial sets satisfying the conditions of \ref{cor:leveltoglobal}. Then $Sd_h f$ and $Sd_v f$ also satisfy the conditions.
\end{lemma}

\begin{proof}
  We treat $Sd_h f$ first. In level $n$ the map $Sd_h f$ is just the map $f \colon X_{2n+1} \to Y_{2n+1}$ which is a homology fibration by assumption. Assume given a vertex $v \in (Sd_hY)_{n,0} = Y_{2n+1,0}$ and a simplicial structure map $\theta \colon [m] \to [n]$. The induced map $ \theta^\ast \colon (SdY)_n \to (SdY)_m$ is the map simplicial structure map $(\theta \sqcup \theta^{op})^\ast \colon Y_{2n+1} \to Y_{2m+1}$ so the map on fibers is a homology equivalence by assumption. 

Now for the map $Sd_vf$. For $n \geq 0$ the map $(Sd_vf)_n$ is the subdivision $Sd(f_n)$ of the map $f_n \colon X_n \to Y_n$, so by \ref{lemma:hfibsd} it is a homology fibration. A vertex $v \in (SdY_n)_0=Y_{n,1}$ need not come from a vertex in $Y_{n,0}$, but it can be connected to such a vertex by an edge. Since $Sd(f_n)$ is a homology fibration it then follows the fiber over $v$ is equivalent to the fiber over a vertex in $Y_n$. This implies that for any simplicial structure map $\theta \colon [m] \to [n]$ the fiber over $v$ maps by a homology equivalence to the fiber over $(Sd\theta^\ast)(v)$.
\end{proof}

\section{Simplicial monoids with anti-involution}\label{section:monoids}

\begin{definition}
  An anti-involution on a monoid $M$ is a function $\alpha \colon M \to M$ such that
  \begin{enumerate}
  \item For all $a,b \in M, \, \alpha(a \cdot b ) = \alpha(b) \cdot \alpha (b)$.
  \item For all $a \in M, \, \alpha(\alpha(m)) = m$.
  \end{enumerate}

A simplicial monoid with anti-involution is a simplicial monoid $M$ with a self-map $\alpha \colon M \to M$ of simplicial sets, which is an anti-involution in each simplicial level. 
\end{definition}
We will often suppress the maps $\alpha$ in the exposition and simply write $\bar{a}$ for $\alpha(a)$. Given a monoid $M$ we can form the bar construction $BM$ which is a simplicial set. If $M$ has the extra structure of an anti-involution we get extra structure on the bar construction as well. The system of maps 
\[\{w_i \colon B_iM \to B_iM \}\]
given in level $p$ by
\[w_p(m_1,m_2, \ldots , m_{p}) = (\overline{m_p},   \ldots,\overline{m_2}, \overline{m_{1}}) \]
together with the simplicial structure maps of $B M$ form a real simplicial set (see the appendix \ref{appendix})
\[B^{1,1} M \colon (\Delta R)^{op} \to Set.\]
Similarly, for a simplicial monoid $M$ with anti-involution we get a functor
 \[B^{1,1} M \colon (\Delta R)^{op} \to sSet.\]
We can extend $M$ to a functor from $ (\Delta R)^{op} \times \Delta ^{op}$ to sets by letting the $w_i$-s act by the involution and the other structure maps in the first factor act trivially. Similarly, we can extend the representable functor $\Delta R^1$ to a functor from $ (\Delta R)^{op} \times \Delta ^{op}$ to sets which is constant in the second variable. Denote the product of these $ (\Delta R)^{op} \times \Delta ^{op}$-sets by $\Delta R^1 \boxtimes M$.

Since the $1$-simplices of  $B^{1,1} M$ are $M$ there is an induced map
\[ \Delta R^1 \boxtimes M \to B^{1,1} M.\]
 As a consequence, we get a $C_2$-equivariant map on realizations 
\[  |\Delta^1| \times |M|   \to |B^{1,1}M| \]
which sends the $\g$-subspace $|\Delta^1|\times \{e\} \cup \{0,1\}\times|M| $ to the basepoint. Therefore, we get an induced $\g$-map $S^{1,1} \wedge |M| \to  |B^{1,1}M|$ with adjoint map
\[\eta_M \colon |M| \to \Omega^{1,1}|B^{1,1}M|. \]
Non-equivariantly, the topological monoid $|M|$ acts by left multiplication on itself and acts homotopy associatively on the loop space by $m \cdot \gamma = \eta(m) \ast \gamma$, where $\ast$ means concatenation of loops. Up to homotopy $\eta$ commutes with the $\g$-actions. We are interested in the properties of the map induced by $\eta$ on fixed points
\[\eta_M^\g \colon |M|^\g \to (\Omega^{1,1}|B^{1,1}M|)^\g. \]
Here, there is a left action of $|M|$ on the fixed points, given by $(m,n) \mapsto mn\bar{m}$. On the fixed points of the loop space $|M|$ acts by $m \cdot \gamma = \eta(m)\ast \gamma \ast \eta(\bar{m})$ and also these actions commute with $\eta_M^\g$ up to homotopy.

\begin{definition}
  Let $N$ be a commutative monoid. An element $s \in N$ is called a cofinal generator if for any $x \in N$ there is an $n \geq 0$ and an element $y \in N$ such that $xy=s^n$. A vertex $t$ in a simplicial monoid $M$ with $\pi_0(M)$ commutative is called a homotopy cofinal generator if its class $[t] \in \pi_0(M)$ is a cofinal generator.
\end{definition}

\begin{example}\label{example:hocof}
  Let $M$ be a simplicial monoid such that the monoid $\pi_0(M)$ is finitely generated and commutative. Pick vertices $t_1, \ldots , t_n \in M_0$ whose path components $[t_1], \ldots , [t_n]$ generate $ \pi_0(M)$. Then the vertex $t = t_1 t_2 \cdots    t_n$ is a homotopy cofinal generator of $M$.
\end{example}

From now on let $M$ denote a simplicial monoid with $\pi_0(M)$ commutative and let $t$ be a homotopy cofinal generator of $M$. For a left $M$-module $X$ we set
\[
X_\infty = \hocolim (X \stackrel{t \cdot}{\too}  X \stackrel{t \cdot}{\too} X \stackrel{t \cdot}{\too} \cdots ).
\]
In particular, we have
\[
M_\infty = \hocolim (M \stackrel{t \cdot}{\too}  M \stackrel{t \cdot}{\too} M \stackrel{t \cdot}{\too} \cdots ).
\]
Multiplication from the left by $t$ commutes with the ordinary right action of $M$ on itself, and there is an induced right action of $M$ on $M_\infty $. The homology of $M$ is a graded ring and the homology of $M_\infty$ is a right module over $H_\ast(M)$. There is an isomorphism of right $H_\ast(M)$-modules
\[
H_\ast (M_\infty) \cong \colim (H_\ast(M) \stackrel{[t] \cdot}{\too}  H_\ast(M) \stackrel{[t] \cdot}{\too} H_\ast(M) \stackrel{[t] \cdot}{\too} \cdots ) 
\]

\begin{lemma}\label{lemma:mincl}The map $M \to M_\infty$ including $M$ at the start of the diagram induces a map on homology 
\[H_\ast(M) \to H_\ast (M_\infty)\]
that sends $\pi_0(M)$ to invertible elements. The induced map
\[H_\ast(M)[\pi_0(M)^{-1}] \to H_\ast(M_\infty)\]
is an isomorphism of right $H_\ast(M)$-modules.
\end{lemma}
\begin{proof}
  Since $\pi_0(M)$ is central in $H_\ast(M)$ there is an isomorphism
\[
\colim (H_\ast(M) \stackrel{[t] \cdot}{\too}  H_\ast(M) \stackrel{[t] \cdot}{\too} H_\ast(M) \stackrel{[t] \cdot}{\too} \cdots ) \cong H_\ast(M)[t^{-1}] 
\]
and since $[t]$ is a cofinal generator of $\pi_0(M)$ there is an isomorphism
\[ H_\ast(M)[t^{-1}] \cong H_\ast(M)[\pi_0(M)^{-1}].\]
\end{proof}

It follows that the vertices $M_0$ of $M$ act on $M_\infty$ by homology equivalences. The following can also be found in e.g., \cite[IV, 5.15]{gj}.

\begin{lemma}\label{lemma:condpf}
  Let $X$ be a right $M$-space on which $M_0$ acts by homology equivalences. Then the canonical map $p \colon B(X , M, \ast) \to BM$ satisfies the conditions of corollary \ref{cor:leveltoglobal}. In particular, it is a homology fibration.
\end{lemma}

\begin{proof}
  First we verify that for each $n \geq 0$ the projection map 
\[p_n \colon X \times M^{\times n} \to M^{\times n} \]
is a homology fibration. Given simplices $\sigma \colon \Delta^p \to M^{\times n}$ and $\tau \colon \Delta^q \to M^{\times n}$ and a map $\alpha \colon \sigma \to \tau$ there is an induced square
\[\xymatrix{p^{-1}(\sigma) \ar[r]^{p^{-1}(\alpha)} \ar[d]_\cong &  p^{-1}(\tau)\ar[d]^\cong \\
\Delta^p \times X \ar[r]_-{\alpha_\ast \times id_X}^-\simeq & \Delta^q \times X,
 }\]
so $p^{-1}(\alpha)$ is a weak equivalence and $p_n$ is a homology fibration.

Next, let $v \in M_{0}^{\times n}$ be a vertex and let $\theta \colon [m] \to [n]$ be a map in $\Delta$. Note that the fiber over any vertex is isomorphic to $X$. We must show that the map on fibers
\[p_n^{-1}(v) \to p_{n-1}^{-1}(\theta^\ast(v)) \]
is a homology equivalence. Since $\theta^\ast$ can be factored into face and degeneracy maps we reduce these cases. If $\theta^\ast = s_i$ or $\theta^\ast = d_j$ with $j \neq 0$ then the map $p_n^{-1}(v) \to p_{n-1}^{-1}(\theta^\ast(v)) $ is an isomorphism. Otherwise, if $\theta^\ast=d_0$, then the induced map is multiplication by an element in $M_0$ and hence a homology equivalence.
\end{proof}

\begin{lemma}\label{lemma:hocolimB}
  Let $G \colon I \to sSet$ be a functor taking values in right $M$-modules and $M$-module maps and let $X$ be a left $M$-module. Then there is a natural isomorphism of simplicial sets
\[ dB(\hocolim G, M, X) \cong \hocolim d B(G,M,X). \]
\end{lemma}

\begin{proof}
Both simplicial sets are obtained by taking iterated diagonals of the trisimplicial set $B(\coprod_\ast G,M,X)$ given by
\[[p],[q],[r] \mapsto \left( \coprod_{\sigma \in N_r(I)}G\left(\sigma\left(0\right)\right)_q\right) \times M^{\times^p}_q \times X_q.\] 
\end{proof}

\begin{corollary}\label{cor:infty}For any left $M$-module $X$ there is an isomorphism 
\[dB(M_\infty,M,X) \cong (dB(M,M,X))_\infty.\]
\end{corollary}

\begin{theorem}[Group completion] (cp. \cite{mcduff-segal},\cite{filtr}\cite{grcomprev}\label{theorem:grcomp})
Let $M$ be a simplicial monoid such that $\pi_0 M$ is in the center of $H_\ast(M)$.   Then there is an isomorphism
\[ H_\ast(M)[\pi_0(M)^{-1}] \iso H_\ast(\Omega|BM|). \]
\end{theorem}

\begin{proof}
Assume first that $M$ has a homotopy cofinal generator $t$. Then, by lemma \ref{lemma:condpf} the map $B(M_\infty,M,\ast) \to BM$ is a homology fibration with fiber $M_\infty$. Taking $X=\ast$ in \ref{cor:infty} shows that the simplicial set $dB(M_\infty,M,\ast)$ is a homotopy colimit of contractible spaces and hence is contractible. From this we get a homology equivalence $|M_\infty| \to \Omega |BM|$ and we conclude by lemma \ref{lemma:mincl}.

For general $M$ we let $F(M)$ denote the poset of submonoids of $M$ with finitely generated monoid of path components. Then there is an isomorphism of simplicial monoids $\colim_{M_i \in F(M)} M_i \cong M$ and the colimit is filtering. The functors $|-|$, $B$, $\Omega$ and $H_\ast(-)$ all commute with filtering colimits so the result now follows since each $M_i \in F(M)$ has a homotopy cofinal generator by \ref{example:hocof}.  
\end{proof}

If we are willing to work a little more we can also see that this isomorphism is $H_\ast(M)$-linear and agrees in homology with the map induced by
\[\eta_M \colon |M| \to \Omega|BM|. \] 
A similar proof  will be given in detail below for the map
\[\eta^\g_M \colon |M|^\g \to (\Omega^{1,1}|B^{1,1}M|)^\g. \]

With theorem \ref{theorem:grcomp} as our starting point we will now proceed to analyse the map $\eta^\g_M$. It becomes easier to work with the anti-involution when we take the Segal subdivision in the horizontal (i.e., bar construction) direction of $B^{1,1}M$. The output is the bisimplicial set $Sd_h B^{1,1}M$ which has a simplicial action of $C_2$ and whose fixed points we will now describe. An element in level $(p,q)$ of $Sd_h B^{1,1}M$ is a tuple
\[(m_1, \ldots , m_{2p+1}) \in M_q^{\times^{2p+1}},\]
and the action of the non-trivial element in $C_2$ is
\[(m_1, \ldots ,m_{p}, m_{p+1},m_{p+2}, \ldots , m_{2p+1}) \mapsto (\overline{m_{2p+1}}, \ldots ,\overline{m_{p+2}},\overline{m_{p+1}},\overline{m_{p}}, \ldots ,\overline{m_1}). \]
The fixed points of this action are of the form
\[(m_1, \ldots ,m_{p}, m_{p+1},\overline{m_{p}}, \ldots ,\overline{m_1}), {\text { where }} m_{p+1}= \overline{m_{p+1}}.\]
Here, the last $p$ factors are redundant and projection on the first $p+1$ factors gives a bijection 
\[b_{p,q} \colon (M_q^{\times^{2p+1}})^\g \iso M_q^{\times^{p}} \times M_q^\g. \]
The monoid $M_q$ acts on $ M_q^\g$ on the left by $(m,n) \mapsto m \cdot n  \cdot  \overline{m}$. Both this action and the description of the fixed points are compatible with the simplicial structure maps of $M$. Combining this with the fact that
\[d_p(m_1, \ldots ,m_{p}, m_{p+1},\overline{m_{p}}, \ldots ,\overline{m_1}) = (m_1, \ldots ,m_{p} \cdot m_{p+1} \cdot \overline{m_{p}}, \ldots ,\overline{m_1}), \]
we get the following:
\begin{lemma}
  Let $M$ be a simplicial monoid with anti-involution. Then the maps $b_{p,q}$ determine natural isomorphism of bisimplicial sets
\[ b \colon (Sd_h B^{1,1}M)^\g \iso B(\ast,M,M^\g).\]
\end{lemma}

The map $p \colon B(M_{\infty},M,\ast) \to  BM$ induces a map
\[Sd_hp \colon Sd_hB(M_{\infty},M,\ast) \to Sd_h BM\]
on subdivisions. Since $p$ satisfies the conditions of \ref{cor:leveltoglobal}, the map $Sd_hp$ does as well, by \ref{lemma:sdbihfib}. Therefore, $Sd_hp$ is a homology fibration.

\begin{lemma}
  The pullback of $Sd_hp$ along the inclusion 
\[B(\ast,M,M^\g) \inj Sd_h BM\]
is isomorphic to $B(M_\infty,M, M^\g)$.
\end{lemma}
The proof is straightforward. It now follows from \ref{lemma:basechange} that the square
\begin{align*}\label{square}\tag{*}\xymatrix{B(M_\infty,M,M^\g) \ar[r] \ar[d] &Sd_hB(M_{\infty},M,\ast) \ar[d]\\
B(\ast,M,M^\g) \ar[r] & Sd_h BM }
\end{align*}
becomes homology cartesian after taking diagonals. We consider $M^\g$ as a bisimplicial set which is constant in the first variable. Define the map
\[i \colon M^\g \to B(M,M,M^\g)\]
levelwise by 
\[m \mapsto (e,e,\ldots, e,m).\]
This map has a retraction $r$ given by 
\[r(m_0,m_1,\ldots, m_p,m) = m_0 \cdot m_1\cdots m_p\cdot m \cdot\overline{m_p}\cdots \overline{m_1}\cdot \overline{m_0}.\]
There is a standard simplicial homotopy $r\circ i \simeq id$. The map $M \to M_\infty$ of \ref{lemma:mincl} induces a map
\[j \colon B(M,M,M^\g) \to B(M_\infty,M,M^\g).\]
\begin{lemma}\label{lemma:ji}
  The map $j \circ i \colon M^\g \to B(M_\infty,M,M^\g)$ induces an isomorphism of left $\pi_0(M)$-sets
\[\pi_0(M^\g)[\pi_0(M)^{-1}] \iso \pi_0(B(M_\infty,M,M^\g)\]
and an isomorphism of $H_\ast(M)$-modules 
\[ H_\ast(M^\g)[\pi_0(M)^{-1}] \iso H_\ast(B(M_\infty,M,M^\g)).\]
\end{lemma}
\begin{proof} We present the argument for homology, the one for $\pi_0$ is similar. By \ref{cor:infty} there is an isomorphism $dB(M,M,M^\g)_\infty \cong B(M_\infty,M,M^\g)$. In the diagram 
\[\xymatrix{dB(M,M,M^\g) \ar[r]^{t \cdot} \ar[d]^{dr} & dB(M,M,M^\g) \ar[r]^{t \cdot} \ar[d]^{dr} & dB(M,M,M^\g) \ar[r]^-{t \cdot} \ar[d]^{dr}  & \cdots\\
M^\g \ar[r]^{t \cdot} & M^\g \ar[r]^{t \cdot} & M^\g \ar[r]^{t \cdot} & \cdots }\]
the vertical maps are weak equivalences and hence induce a weak equivalence of homotopy colimits $dB(M,M,M^\g)_\infty \stackrel{r_\infty}{\too} M^\g_\infty$. In homology we get a sequence of isomorphisms of left $H_\ast(M)$-modules
\[H_\ast(B(M_\infty,M,M^\g)) \iso H_\ast(M^\g_\infty)  \iso H_\ast(M^\g)[\pi_0(M)^{-1}].\]

\end{proof}
 
Let $(X,x)$ be a based $\g$-space with $\sigma \colon X \to X$ representing the action of the non-trivial element of $\g$. The homotopy fiber $hF_{\iota_X}$ of the canonical inclusion
\[\iota_X \colon X^\g \inj X \]
of the fixed points is (homeomorphic to) the space of paths $\chi \colon [0,\tfrac{1}{2}] \to X$ such that $\chi(0) = x$ and $\chi(\tfrac{1}{2})\in X^\g$. There is a map
\[b_X \colon hF_{\iota_X}  \to (\Omega^{1,1}X)^\g\]
given by $b_X(\chi) = \chi \ast (\sigma \circ \overline{ \chi})$ where $\ast$ is the concatenation operation and $\overline{ \chi}$ is the path $t \mapsto \chi(1-t)$. This map is a homeomorphism with inverse given by restricting loops to $[0,\tfrac{1}{2}]$.

Now we apply geometric realization to the square (\ref{square}) to obtain a homology cartesian square of spaces
\[\xymatrix{|B(M_\infty,M,M^\g)| \ar[r] \ar[d] &|B(M_{\infty},M,\ast)| \ar[d]\\
|B(\ast,M,M^\g)| \ar[r] & | BM|. }\]
The space $|B(M_\infty,M, \ast)|$ is contractible and so $|B(M_\infty,M,M^\g)|$ is homology equivalent to the homotopy fiber of the composite
\[|B(\ast,M,M^\g)| \cong |B^{1,1}M|^\g \inj |BM|.\]

By the discussion above, this space is homeomorphic to $(\Omega^{1,1}|B^{1,1}M|)^\g$ and so we get a homology equivalence
\[g \colon |B(M_\infty,M,M^\g)| \to (\Omega^{1,1}|B^{1,1}M|)^\g.\]

\begin{theorem}\label{theorem:mainmon}
  Let $M$ be a simplicial monoid with anti-involution such that $\pi_0 M$ is in the center of $H_\ast(M)$. Then the map
\[\eta^\g_M \colon M^\g \to (\Omega^{1,1}|B^{1,1}M|)^\g\]
induces an isomorphism 
\[\pi_0(M^\g)[\pi_0(M)^{-1}] \iso \pi_0(\Omega^{1,1}|B^{1,1}M|)^\g\]
of $\pi_0(M)$-sets and an isomorphism of $H_\ast(M)$-modules 
\[H_\ast(M^\g)[\pi_0(M)^{-1}] \iso H_\ast((\Omega^{1,1}|B^{1,1}M|)^\g).\]
\end{theorem}
\begin{proof}The proof of the $\pi_0$-statement follows the same outline as the homology-statement but is easier and is therefore omitted. 

For the homology-statement we first assume that $M$ has a homotopy cofinal generator $t \in M_0$. It remains to be shown that the map $g_\ast$ induced on homology by $h$ agrees with the map induced by $\eta^\g_M$. Since $g$ is induced by a contracting homotopy of $|B(M_\infty,M,\ast)|$ we must investigate such homotopies. Any two contracting homotopies of $|B(M_\infty,M,\ast)|$ will be homotopic, so it will suffice to find some homotopy that does what we want. We write $B(|M|,|M|,\ast)$ for the realization of the two-sided bar construction on the topological monoid $|M|$ with the $|M|$-spaces $|M|$ and $\ast$ and similarly for $B|M| = B(\ast,|M|,\ast)$. There are natural homeomorphisms $B(|M|,|M|,\ast) \cong |B(M,M,\ast)|$ and $B|M| \cong |BM|$. A contracting homotopy $h$ for $B(|M|,|M|, \ast )$ given by 
\[h_u(m_0, \ldots , m_k, t_0 , \ldots , t_k) = (e,m_0, \ldots , m_k, 1-u, ut_0 , \ldots , ut_k).\]
The space $B(|M|,|M|, \ast )$ sits inside $|B(M_\infty,M,\ast)|$ as the subcomplex over the initial object of the diagram for $M_\infty$. We can choose a contracting homotopy that extends $h$ to all of $|B(M_\infty,M,\ast)|$. The space $|M^\g|$ maps into $B(|M|,|M|, \ast )$ by 
\[m \mapsto (e,m, \tfrac{1}{2},\tfrac{1}{2})\]
which maps to $(m, \tfrac{1}{2},\tfrac{1}{2})\in B|M|.$ The path traced out in $B|M|$ by an element $m \in |M^\g|$ is therefore
\[u \mapsto (e,m,(1-u),\tfrac{u}{2},\tfrac{u}{2})\]
which is an element of the homotopy fibre of the inclusion $|BM|^\g \inj |BM|$. The map $b_{|BM|}$ to the equivariant loop space $(\Omega^{1,1}|B^{1,1}M|)^\g)$ sends this element to $\eta^\g_M(m)$.

For a general $M$ we reduce to the above case by a colimit argument as in the proof of theorem \ref{theorem:grcomp}. 
\end{proof}

\section{Categories with duality}
In this section we summarize some facts we will need later. Again we make no claim of originality and the reader can consult \cite{dotto}, \cite{schlichting} or \cite{iblars} for details.
\begin{definition} A category with duality is a triple $(\cat{C}, T, \eta)$ where $\cat{C}$ is a category, $T \colon \cat{C}^{op} \to \cat{C}$ is a functor and $\eta \colon  id \to  T\circ T^{op}$ is an isomorphism of functors such that for each $c$ in $\cat{C}$ the composite map 
\[\xymatrix{Tc \ar[r]^-{\eta_{Tc}} & TT^{op}Tc \ar[r]^-{T(\eta_c)} &Tc } \] 
is the identity. If $\eta = id$, so that $T\circ T^{op} = Id_\cat{C}$, then the duality is said to be strict.
\end{definition}

\begin{example}
  A monoid $M$ can be thought of as a category $\cat{C}_M$ with one object $\ast$ and $Hom_{\cat{C}_M} (\ast,\ast) = M$ as monoids. Then a duality $T,\eta$ on $\cat{C}_M$ is the same as a monoid map $t \colon M^{op} \to M$ and an invertible element $\eta \in M$ such that $\eta t^2(m) = \eta m$ and $t(\eta) = \eta^{-1}$. The duality is strict if and only if $t$ is an anti-involution on $M$.
\end{example}

The main example of interest to us is the following (see e.g. \cite{wall}).
\begin{example} A Wall anti-structure is a triple $(R,\alpha,\varepsilon)$ where $R$ is a ring, $\alpha$ is an additive map $R \to R$ such that $\alpha(rs) = \alpha(s)\alpha(r)$ and $\varepsilon$ is a unit in $R$ such that $\alpha^2(r) = \varepsilon r \varepsilon^{-1}$ and $\alpha(\varepsilon) =  \varepsilon^{-1}$. For an anti-structure $(R,\alpha,\varepsilon)$ there is a naturally associated category with duality $P(R,\alpha,\varepsilon)$ with underlying category $P(R)$ the category of finitely generated projective (f.g.p) right $R$-modules. The duality functor on $P(R,\alpha,\varepsilon)$ is $Hom_R(-,R)$ where for an f.g.p. module $P$ we give $Hom_R(P,R)$ the right (!) module structure given by $(fr)(p) = \alpha(r) f(p)$. The isomorphism 
\[ \eta_P \colon P \iso Hom_R(Hom_R(P,R),R)\]
is given on elements $p \in P$ by $\eta_P(p)(f) = \alpha(f(p))\varepsilon$. It is straightforward to check that the equation $\eta_{Hom_R(P,R)} \circ \eta_P^\ast = id_{Hom_R(P,R)}$ holds for all f.g.p. modules $P$. 
\end{example}

\begin{definition}
  A duality preserving functor
\[ (F, \xi)  \colon (\cat{C}, T, \eta) \too (\cat{C}', T', \eta') \]
consists of a functor $F \colon \cat{C} \to \cat{C}'$ and a natural transformation
\[ \xi \colon F \circ T \to T' \circ F\]
 such that for all $c$ in $\cat{C}$
the diagram
\[ \xymatrix{ F(c) \ar[r]^{\eta'_{F(c)}} \ar[d]_{F(\eta_c)} &  T' \circ (T')^{op} \circ F \circ (c) \ar[d]^{T'(\xi_c)} \\
F \circ T \circ T^{op} (c) \ar[r]_{\xi_{T(c)}} & T' \circ F^{op} \circ T^{op} (c)}\]
commutes.
\end{definition}
Composition is given by $(G,\zeta) \circ (F,\xi)= (G \circ F,\zeta_F\circ G(\xi))$. A duality preserving functor $(F, \xi) \colon (\cat{C}, T, \eta) \too (\cat{C}', T', \eta')$ is called an equivalence of categories with duality if there is a duality preserving functor $(F', \xi') \colon (\cat{C}, T, \eta) \too (\cat{C}', T', \eta')$ and natural isomorphisms $u \colon F' \circ F \iso Id_\cat{C}$ and $u' \colon F\circ F' \iso Id_\cat{C'}$. The transformation $u$ must additionally satisfy $\xi'_{F(c)} \circ F'(\xi_c) = T(u_c)\circ u_{T(c)}$ and similarly for $u'$. 
\begin{definition}
  Let $(\cat{C},T,\eta)$ be a category with duality. The category $Sym(\cat{C},T,\eta)$ of symmetric forms in $(\cat{C},T,\eta)$ is given as follows:
  \begin{itemize}
  \item The objects of $Sym(\cat{C},T,\eta)$ are maps $f \colon a \to Ta$ such that $f = Tf\circ \eta_a$.
\item A morphism from $f \colon a \to Ta$ to $f' \colon a' \to Ta'$ is a map $r \colon a \to a'$ in $\cat{C}$ such that the diagram
\[ \xymatrix{a \ar[r]^f \ar[d]_r & Ta  \\
a'  \ar[r]_{f'} & Ta' \ar[u]_{Tr} }\]
commutes.
\item Composition is given by ordinary composition of maps in $\cat{C}$.
  \end{itemize}
\end{definition}

The reason for the name ``symmetric form'' in the preceding definition is the following. Let $(R,\alpha,\varepsilon)$ be a Wall-anti-structure. The category $SymP(R,\alpha,\varepsilon)$ has as objects maps $\varphi \colon P \to Hom_R(P,R)$ such that the adjoint map $\tilde{\varphi} \colon P \otimes_\Z P \to R$ is a biadditive form on $P$ satisfying
\begin{align*}
  \tilde{\varphi}(pr,qs) &= \alpha(r)\tilde{\varphi}(p,q)s\\
\tilde{\varphi}(q,p) &= \alpha(\tilde{\varphi}(p,q))\varepsilon,
\end{align*}
for $r,s \in R$ and $p,q \in P$. A map 
\[h \colon (P \stackrel{\varphi}{\too} Hom_R(P,R)) \to (P' \stackrel{\varphi'}{\too} Hom_R(P',R))\]
is an $R$-module homomorphism $h \colon P \to P'$ such that $\tilde{\varphi}'(h(p),h(q)) = \tilde{\varphi}(p,q)$ for all $p,q \in P$. An object $\varphi \colon P \to Hom_R(P,R)$ such that $\varphi$ is an isomorphism is called non-degenerate. 

\begin{definition}
For a category with duality $  (\cat{C}, T, \eta)$ the category $\cat{D}(\cat{C}, T, \eta)$ has objects triples $(c,c',f)$ where $f \colon c' \iso Tc$ is an isomorphism and maps from $(c,c',f)$ to $(d,d',g)$ are pairs $(r \colon c \to d,s \colon d' \to c' )$ such that the diagram
\[\xymatrix{ c' \ar[r]^f  & Tc \\
  d' \ar[r]_g \ar[u]^s & Td \ar[u]_{Tr}}\]
commutes. Composition is given by composition in each component. The duality on $\cat{D}(\cat{C}, T, \eta)$ is given by sending an object $f \colon c' \iso Tc$ to the composite $c \stackrel{\eta_c}{\too} TT^{op}c \stackrel{Tf}{\too} Tc'$ and $(r \colon c \to d,s \colon d' \to c' )$ to $(s \colon d' \to c', r \colon c \to d )$. 
\end{definition}
It is easy to see that the duality on $\cat{D}(\cat{C}, T, \eta)$ is strict. The functor
\[(I,\iota) \colon (\cat{C}, T, \eta) \to \cat{D}(\cat{C}, T, \eta)\]
given by $I(c) = (c,Tc,id_{T(c)})$, $I(f) = (f,Tf)$ and $\iota_c = (id_{T(c)},\eta_c)$ is  an equivalence of categories with duality. Its inverse is the duality preserving functor
\[(K,\kappa) \colon \cat{D}(\cat{C}, T, \eta) \to (\cat{C}, T, \eta),\]
given by $K(c,c',f) = c$, $K(r,s) = r$ and $\kappa_{(c,c',f)} = f$. Both the construction $\cat{D}$ and the functors $K$ and $I$ are functorial in $(\cat{C}, T, \eta)$ for duality preserving functors. 

From the topological perspective the effect of the $\cat{D}$-construction is to replace the geometric realization $|N\cat{C}|$, which has an action of $\g$ in the homotopy category, by the bigger space $|N\cat{D}\cat{C}|$ which has a continuous action of $\g$. The actions are compatible in the sense that the maps $|NI|$ and $|NK|$ are mutually inverse isomorphisms of $\g$-objects in the the homotopy category. Observe that a strict duality $T$ on a category $\cat{C}$ gives a map 
\[NT \colon (N\cat{C})^{op} = N(\cat{C}^{op}) \iso N\cat{C}\]
such that $NT \circ (NT)^{op} = id_{N\cat{C}}$. We know from \ref{lemma:realstr} that this is equivalent to extending the simplicial structure of $N\cat{C}$ to a real simplicial structure. It follows that the realization has an induced $\g$-action given by \[[(c_0 \stackrel{f_1}{\too}  \ldots \stackrel{f_n}{\too} c_n ,t_0, \ldots t_n)] \mapsto [(c_n \stackrel{Tf_n}{\too} \ldots \stackrel{Tf_1}{\too} Tc_0,t_n, \ldots t_0)], \]
for $(c_0 \stackrel{f_1}{\too}  \ldots \stackrel{f_n}{\too} c_n ,t_0, \ldots t_n) \in N_n\cat{C}\times \Delta^n$.
 
\begin{definition}
  Let $\cat{C}$ be a category. Its subdivision $Sd\cat{C}$ is a category given as follows: An object of $Sd\cat{C}$ is a morphism $f \colon a \to b$ in $\cat{C}$ and a map from $f \colon a \to b$ to $g \colon c \to d$ is a pair $(h,i)$ of maps such that the following diagram commutes
\[\xymatrix{a \ar[r]^f\ar[d]_h &   b  \\
c \ar[r]_g   & \ar[u]_id.
}\]
Composition is given by $(h',i') \circ (h,i) = (h' \circ h,i \circ i')$.
\end{definition}
Note that $SdN\cat{C} = NSd\cat{C}$. If $(\cat{C},T,\eta)$ is a category with duality then there is an induced functor 
\[SdT \colon Sd\cat{C} \to Sd\cat{C}\]
given by $SdT(a \stackrel{f}{\too} b) = Tb \stackrel{Tf}{\too} Ta$ and $SdT(h,i) = (Ti,Th)$. If $T$ is a strict duality then $Sym(\cat{C})$ is the category fixed under the $\g$-action defined by $SdT$. The functor 
\[(I,\iota) \colon (\cat{C},T,\eta) \to \cat{D}(\cat{C},T,\eta)\]
induces an equivalence of categories $Sym(I)\colon Sym(\cat{C}) \to Sym(\cat{D}\cat{C})$, so up to equivalence, $Sym(\cat{C})$ can be thought of as a fixed category under a strict duality.

\section{$K$-theory of additive categories with duality}
\begin{definition}
  Let $\cat{C}$ be a category and let $X$ and $X'$ be objects of $\cat{C}$. A biproduct diagram for the pair $(X,X')$ is a diagram
\begin{equation}\label{biproduct}\xymatrix{X \ar@<-2pt>[r]_{i_1} & Y \ar@< 2pt>[r]^{p_2} \ar@<-2pt>[l]_{p_1} & X' \ar@< 2pt>[l]^{i_2}}
\end{equation}
in $\cat{C}$ such that the $p_j$-s express $Y$ as the product of $X$ and $X'$ and the $i_j$-s express $Y$ as a coproduct of $X$ and $X'$ 
\end{definition}

If $\cat{C}$ is a category which has a zero object and each pair of objects has a biproduct diagram in $\cat{C}$ the hom-sets of $\cat{C}$ naturally inherit the structure of commutative monoids such that composition is bilinear \cite[VIII,2]{categories}. We call such a category $\cat{C}$ \emph{additive} if the hom-sets are abelian groups, not just monoids. A functor between additive categories is called additive if it preserves biproducts. Additive functors induce group homomorphisms on hom-groups.

Let $X$ be a finite pointed set. The category $Q(X)$ is defined as follows: The objects in $Q(X)$ are the pointed subsets $U \subseteq X$. A morphism $U \to V$ of pointed subsets is a pointed subset of the intersection $U \cap V$. The composition of two subsets $A \subseteq U \cap V$ and $B \subseteq V \cap W$ is $A \cap B \subseteq U \cap W$. Note that $A \subseteq U \cap V$ can be thought of both as a map from $U$ to $V$ and as a map from $V$ to $U$, in fact $Q(X) = Q(X)^{op}$. 

\begin{definition}
  Let $\cat{C}$ be an additive category and $X$ a finite pointed set. A sum-diagram in $\cat{C}$ indexed by $X$ is a functor
\[A \colon Q(X) \to \cat{C}\]
such that for any subset $U \subseteq X$ the maps
\[A(U) \to A(\{u,\ast\})\]
induced by the pointed subsets $\{u,\ast\} \subseteq U$, exhibit $A(U)$ as the product of the $A(\{u,\ast\})$'s. 
\end{definition}
This is equivalent to thinking of sum-diagrams as sheaves on $Q(X)$ endowed with a certain Grothendieck topology, see \cite[4.3]{iblars} for details. 

A \emph{pointed} category is an category $\cat{C}$ with a chosen object $0_{\cat{C}}$. Wen $\cat{C}$ is additive $0_{\cat{C}}$ will always be a zero-object, but in general it need not be. We say that a functor between pointed categories is pointed if it preserves the chosen objects. Many of the constructions we will do in the following rely on having chosen basepoints. To avoid confusion and to make our constructions functorial we will usually work with pointed categories. 

For a pointed additive category and a finite pointed set $X$ we require that the elements of $\cat{C}(X)$ be pointed, i.e., that the send the subset $\{\ast\}$ to $0_{\cat{C}}$. We write $\cat{C}^X$ for the (pointed) category $Fun_\ast(X,\cat{C})$ of pointed functors from $X$ to $\cat{C}$, where we think of $X$ as a discrete category. There is a natural evaluation functor $e_X \colon \cat{C}(X) \to \cat{C}^X$ given on objects by $e_X(A)(x) = A(\{x,\ast\})$ and similarly for morphisms.
\begin{lemma}\label{lemma:ev}
  Let $\cat{C}$ be a pointed additive category. For any finite pointed set $X$ the functor
\[ e_X \colon  \cat{C}(X) \to \cat{C}^X\]
is an equivalence of categories.
\end{lemma}

A pointed map $f \colon X \to Y$ induces a pushforward functor $f_\ast \colon \cat{C}(X) \to \cat{C}(Y)$ given by 
\[(f_\ast(A))(U) = A(f^{-1}(U\setminus\{\ast\})\cup\{\ast\})).\]  
Given two composable maps $f$ and $g$ of finite pointed sets it is not hard to see that $(f\circ g)_\ast = f_\ast \circ g_\ast$, so that we get a functor 
\[\cat{C}(-) \colon FinSet_\ast \to Cat_\ast,\]
where $FinSet_\ast$ is the category of finite sets and pointed maps and $Cat_\ast$ is the category of small pointed categories and pointed functors between them. This notion coincides up to suitable equivalence with Segal's $\Gamma$-category construction \cite{catcoh}. If $S$ is a pointed simplicial set which is finite in each simplicial level we can regard it as a functor $S \colon \Delta^{op} \to FinSet_\ast$ and form the composite functor $\cat{C}(S)$ which is a simplicial pointed category, i.e. a simplicial object in $Cat_\ast$.

\begin{definition}An additive category with weak equivalences is a pair $(\cat{C},w\cat{C})$ where $\cat{C}$ is an additive category and $w\cat{C} \subseteq \cat{C}$ is a subcategory such that
  \begin{itemize}
  \item all isomorphisms are in $w\cat{C}$
\item if $f$ and $g$ are in $w\cat{C}$ then their coproduct $f \oplus g$ is in $w\cat{C}$. 
  \end{itemize}

\end{definition}

A map $F \colon (\cat{C},w\cat{C}) \to (\cat{C}',w'\cat{C}')$ additive categories with weak equivalences is an additive functor on the underlying categories which preserves weak equivalences. It is an equivalence of additive categories with weak equivalences if the underlying functor is an equivalence and any inverse of it preserves weak equivalences. If $\cat{C}$ is pointed we take $w\cat{C}$ to be pointed with the same chosen object as $\cat{C}$.

Let $(\cat{C},w\cat{C})$ be a pointed additive category with weak equivalences and $X$ a finite pointed set. The subcategory $w\cat{C}(X) \subseteq \cat{C}(X)$ which has the same objects as $\cat{C}(X)$ and morphisms that are pointwise in $w\cat{C}$ is a subcategory of weak equivalences. If $f \colon X \to Y$ is a pointed map, the functor $f_\ast$ sends $w\cat{C}(X)$ into $w\cat{C}(Y)$, so there is an induced functor  
\[w\cat{C}(-) \colon FinSet_\ast \to Cat_\ast.\]
As in lemma \ref{lemma:ev} the functor $we_X \colon w\cat{C}(X) \to w\cat{C}^X$ induced by $e_X$ is an equivalence of categories. We write $S^1$ for the simplicial circle $\Delta^1/\partial \Delta^1$, with basepoint $[\partial \Delta^1]$. The space $\Omega|Nw\cat{C}(S^1)|$ is a model for the algebraic $K$-theory of $(\cat{C},w\cat{C})$, analogous to the space $\Omega |BM|$ for a simplicial monoid $M$. 

The functor $w\cat{C} \to w\cat{C}(S^1_1)$ sending an object $c$ to the diagram with value $c$ on the non-trivial subset of $S^1_1$ and $0_{\cat{C}}$ on $\{\ast\}$ is an equivalence of categories. There is an induced map
\[\Delta^1 \boxtimes Nw\cat{C} \to Nw\cat{C}(S^1)\]
of bisimplicial sets which induces a map
\[\eta_{\cat{C}} \colon |Nw\cat{C}| \to \Omega|Nw\cat{C}(S^1)|\]
of spaces. In \cite[§4]{catcoh} Segal proves a group completion theorem for the map $\eta_{\cat{C}}$ analogous to \ref{theorem:grcomp}. We will mimic the treatment of the monoid case above to reprove Segal's result and extend it to an equivariant statement analogous to \ref{theorem:mainmon} in the case that $\cat{C}$ has an additive duality. 
\begin{lemma} (see e.g. \cite[5.11]{iblars})\label{lemma:strictify} Let $(\cat{C},w\cat{C})$ be an additive category with weak equivalences. Then there is a pointed additive category with weak equivalences $(\cat{C}',w'\cat{C}')$ and an additive equivalence $F \colon (\cat{C},w\cat{C}) \to (\cat{C}',w'\cat{C}')$ such that $(\cat{C}',w'\cat{C}')$ has a coproduct functor 
\[\oplus \colon \cat{C}' \times \cat{C}' \to \cat{C}'\]
 making $\cat{C}'$ a strictly unital, strictly associative symmetric monoidal category. 
\end{lemma}
The construction $w\cat{C}(S^1)$ makes sense also for non-pointed $\cat{C}$ but one must choose a basepoint for $\Omega|Nw\cat{C}(S^1)$ and  $\eta_\cat{C}$ to be defined. This can be done is such a way that the induced map $F_{S^1} \colon w\cat{C}(S^1) \to w'\cat{C}'(S^1)$ gives a homotopy equivalence on realizations and there is a commutative diagram
\[\xymatrix{|Nw\cat{C}| \ar[d]_F \ar[r]^-{\eta_\cat{C}} &  \Omega|Nw\cat{C}(S^1)| \ar[d]^{\Omega|NF_{S^1}|}\\
|Nw'\cat{C}'|  \ar[r]_-{\eta_\cat{C}'} &  \Omega|Nw'\cat{C}'(S^1)|}\]
of spaces, in which the vertical maps are homotopy equivalences and H-maps. From now on we assume, without loss of generality, that $(\cat{C},w\cat{C})$ is pointed and has a coproduct functor $\oplus$ as in lemma \ref{lemma:strictify}. 

The path components of the nerve $Nw\cat{C}$ will be called weak equivalence classes. The set $\pi_0Nw\cat{C}$ of such classes is a commutative monoid under the operation $[a]+[b] = [a \oplus b]$. We assume that the $\pi_0Nw\cat{C}$ has a cofinal generator represented by an object $t$ of $\cat{C}$. Then there is a functor $t \oplus - \colon \cat{C} \to \cat{C}$ which restricts to an endofunctor on $w\cat{C}$. By analogy with the monoid case above we form the diagram
\[ w\cat{C} \stackrel{t \oplus -}{\too}  w\cat{C} \stackrel{t \oplus -}{\too} w\cat{C} \stackrel{t \oplus -}{\too} \cdots \]
of categories. We define $\underline{\mathbb{N}}$ to be the category generated by the graph 
\[0 \to 1 \to 2 \to \cdots\]
so that the above diagram of categories becomes a functor $D \colon \underline{\mathbb{N}} \to Cat$ in the obvious way. Now set $w\cat{C}_\infty = \underline{\mathbb{N}} \wr D$, where $\wr$ denotes the Grothendieck construction (see e.g. \cite{thomason}). The objects of the category  $w\cat{C}_\infty$ are pairs $(m,c) \in \mathbb{N} \times ob\cat{C}$ and a map $(n,c) \to (n+k,d)$ is a map $(t\oplus-)^k(c) \to d$ in $w\cat{C}$. Thomason \cite[1.2]{thomason} constructs a natural weak equivalence
\[\hocolim(ND) \to Nw\cat{C}_\infty.\]
Since the nerve $Nw\cat{C}$ is a simplicial monoid, its homology $H_\ast(Nw\cat{C})$ is a ring under the induced Pontrjagin product. The following is a special case of \ref{lemma:mincl}.
\begin{lemma}\label{lemma:cathloc}
  The canonical functor $w\cat{C} \to w\cat{C}_\infty$ sending an object $c$ to $(0,c)$ induces an isomorphism
\[H_\ast(Nw\cat{C})[\pi_0(Nw\cat{C})^{-1}] \iso H_\ast(Nw\cat{C}_\infty)\]
of right $H_\ast(Nw\cat{C})$-modules.
\end{lemma}
We now recall the simplicial path construction (see \cite[1.5]{waldhausen}). Define the shift functor
\[P \colon \Delta \to \Delta\]
by $P([n]) = [0] \sqcup [n] = [n+1]$ and $P(\alpha) = id_{[0]}\sqcup \alpha$. For a simplicial object $X \colon \Delta^{op} \to \cat{A}$ the (simplicial) path object $PX$ on $X$ is defined as $PX = X \circ P^{op}$. The natural transformation $\delta^0 \colon Id_\Delta \to P$ given on objects by $\delta^0 \colon [n] \to [n+1]$ gives a natural map $d_0 \colon PX \to X$. For a simplicial set $X$ there is a natural map $PX \to X_0$ onto the vertices of $X$ which is a simplicial homotopy equivalence \cite[1.5.1]{waldhausen}. In the case of the simplicial circle the map $d_ 0 \colon PS^1 \to S^1$ induces a map $w\cat{C}(PS^1) \to w\cat{C}(S^1)$ of simplicial categories which we will also call $d_0$. There is a simplicial homotopy equivalence $PS^1 \stackrel{\simeq}{\too} \ast$ which induces a weak equivalence $Nw\cat{C}(PS^1) \stackrel{\simeq}{\too} Nw\cat{C}(\ast) \simeq \ast$ of bisimplicial sets. Let $\zeta_n \in \Delta^1_n$ be the element such that $ \zeta_n(0) = 0$ and $\zeta_n(i) = 1$ for $i \geq 1$. We denote its image in the quotient set $\Delta^1_n /\partial \Delta^1_ n$ by $z_n$ and write $\tilde{c}_{n+1}$ for the diagram in $\cat{C}((PS^1)_n)=\cat{C}(S^1_{n+1})$ whose value is $c \in ob\cat{C}$ on all pointed subsets containing $z_{n+1}$ and $0_\cat{C}$ on the other subsets. The maps between $c$'s in $\tilde{c}_{n+1}$ are all identities and the remaining maps are zero. The functor $d_0 \colon \cat{C}(S^1_{n+1}) \to \cat{C}(S^1_n)$ restricts diagrams to the part away from $z_{n+1}$, so $d_0(\tilde{c}_{n+1}) = 0_{w\cat{C}(S^1_n)}$, the $0$-diagram. Now set $c =t$. Adding the object $\tilde{t}_{n+1}$ from the left gives a functor
\[\tilde{t}_{n+1} \oplus - \colon w\cat{C}(S^1_{n+1}) \to w\cat{C}(S^1_{n+1}). \]
We define $w\cat{C}(S^1_{n+1})_\infty$ to be the Grothendieck construction on the diagram
\[ w\cat{C}(S^1_{n+1})  \stackrel{\tilde{t}_{n+1} \oplus -}{\too}  w\cat{C}(S^1_{n+1})  \stackrel{\tilde{t}_{n+1} \oplus -}{\too} w\cat{C}(S^1_{n+1})  \stackrel{\tilde{t}_{n+1} \oplus -}{\too} \cdots. \]
Since $0$ is a strict unit in $\cat{C}$ the system functors $\{\tilde{t}_{n+1}\oplus-\}_{n \geq 0}$ commutes with the structure maps of $w\cat{C}(PS^1)$, and the map $d_0 \colon w\cat{C}(PS^1) \to w\cat{C}(S^1)$. Therefore the $w\cat{C}(S^1_{n+1})_\infty$'s assemble to a simplicial category $w\cat{C}(PS^1)_\infty$ with a map $d_{0,\infty} \colon w\cat{C}(PS^1)_\infty \to  w\cat{C}(S^1)$. The inclusion of $w\cat{C}(S^1_{n+1})$ in the first spot of the diagram gives a map $w\cat{C}(PS^1) \to w\cat{C}(PS^1)_\infty$ such that the diagram
\[\xymatrix{w\cat{C}(PS^1) \ar[rr] \ar[dr]_{d_0} & & w\cat{C}(PS^1)_\infty \ar[dl]^{d_{0,\infty}}\\
&w\cat{C}(S^1)& }\]
commutes.
\begin{proposition}
  The induced map on nerves 
\[Nd_{0,\infty} \colon N w\cat{C}(PS^1)_\infty  \to N w\cat{C}(S^1)\]
is a homology fibration of bisimplicial sets.
\end{proposition}
\begin{proof}
  We will show that the map satisfies the conditions of \ref{cor:leveltoglobal}. First, we verify that it is a levelwise homology fibration. The pointed set $S^1_{n+1}$ has $n+1$ non-basepoint elements and evaluation gives an equivalence of categories $we_{n+1} \colon w\cat{C}(S^1_{n+1}) \to w\cat{C}^{\times^{n+1}}$. It commutes with the functor $\tilde{t}_{n+1}\oplus-$ in the sense that the following diagram commutes
\[\xymatrix{w\cat{C}(S^1_{n+1}) \ar[r]^-{\tilde{t}_{n+1}\oplus-} \ar[d]_{we_{n+1}} & w\cat{C}(S^1_{n+1}) \ar[d]^{we_{n+1}} \\
w\cat{C}^{\times^{n+1}} \ar[r]_-{(t\oplus-)\times id_{w\cat{C}}^{\times^{n}}} & w\cat{C}^{\times^{n+1}}.}\]
There is an induced equivalence of categories $w\cat{C}(S^1_{n+1})_\infty \to w\cat{C}_\infty \times w\cat{C}^{\times^{n}}$ giving a commutative diagram
\begin{equation}\label{eqn:inftydiag}
\xymatrix{w\cat{C}(S^1_{n+1})_\infty \ar[r]^{e_{n,\infty}} \ar[d] & w\cat{C}_\infty \times w\cat{C}^{\times^{n}} \ar[d]^{p} \\
w\cat{C}(S^1_{n}) \ar[r]_{e_n} & w\cat{C}^{\times^{n}},}
\end{equation}
where the horizontal arrows are equivalences of categories. A simplex $\sigma \colon \Delta^m \to Nw\cat{C}(S^1_{n})$ comes from a uniquely determined functor $\sigma \colon [m] \to w\cat{C}(S^1_{n})$, where $[m]$ is the poset category $0 \to 1 \to \cdots \to m$. We define $(d_{0,\infty})_n^{-1}(\sigma)$ to be the pullback in the diagram
\[\xymatrix{(d_{0,\infty})_n^{-1}(\sigma) \ar[d] \ar[r] & w\cat{C}(S^1_{n+1})_\infty \ar[d]^{(d_{0,\infty})_n}\\
[m] \ar[r]_\sigma & w\cat{C}(S^1_{n}),}\]
and $pr^{-1}(e_n\circ \sigma)$ similarly. We claim that the functor 
\[e_n^{-1}(\sigma) \colon (d_{0,\infty})_n^{-1}(\sigma) \to p^{-1}(e_n\circ \sigma)\]
induced by $e_n$ and $e_{n,\infty}$ is an equivalence of categories. It is easily seen to be surjective on objects, and we must show that it is also fully faithful. An object in $(d_{0,\infty})_n^{-1}(\sigma)$ is a pair $(i,(X,n))$ such that $\sigma(i) = d_0X$. A morphism from $(i,(X,n))$ to $(j,(Y,n+k))$ in $(d_{0,\infty})_n^{-1}(\sigma)$ consists of the map $i \leq j$ and a map $f \colon (\tilde{t}_{n+1}\oplus-)^kX \to Y$ such that $\sigma(i \leq j) = d_0(f)$. Such a map is uniquely determined by what it does on the subsets $\{x,\ast\}$ of $S^1_{n+1}$, and a tuple of maps determines a unique map of diagrams, so the functor $e_n^{-1}(\sigma)$ is fully faithful. Consider the diagram
\[\xymatrix{\Delta^m \ar[r]^-\sigma \ar[d]_{id} & Nw\cat{C}(S^1_n) \ar[d]^{Ne_n} & Nw\cat{C}(S^1_{n+1})_\infty  \ar[d]^{Ne_{n,\infty}} \ar[l]_-{N(d_{0,\infty})_n} \\
\Delta^m \ar[r]_-{N(e_n \circ \sigma)} & Nw\cat{C}^{\times^{n}} & Nw\cat{C}_\infty \times Nw\cat{C}^{\times^{n}} \ar[l]
}\]
of simplicial sets. Since the vertical maps are weak equivalences the induced map on homotopy pullbacks is a weak equivalence. We have just seen that the map on pullbacks is a weak equivalence, and the map from the pullback to the homotopy pullback in the lower part of the diagram is a weak equivalence. It follows that the same holds for the upper part. Now since this holds for all simplices $\sigma$ the map $N(d_{0,\infty})_n$ is a homology fibration. 

To see that the second condition of \ref{cor:leveltoglobal} holds we observe that the fiber over an object $X$ in $w\cat{C}(S^1_{n})$ is equivalent to $w\cat{C}_\infty$. Now we conclude by \ref{lemma:cathloc} in the same way as in the proof of \ref{lemma:condpf}.
\end{proof}
The proof of the following theorem (cp. \cite[5]{catcoh}, \cite[Q.9]{filtr}) is similar to that of \ref{theorem:grcomp}.
\begin{theorem}[$K$-theoretic group completion]
   The map $\eta_\cat{C}$ induces an isomorphism of $H_\ast(Nw\cat{C})$-modules
\[H_\ast(Nw\cat{C})[\pi_0(Nw\cat{C})^{-1}] \iso H_\ast(\Omega|Nw\cat{C}(S^1)|).\]
\end{theorem}

We now turn to additive categories with duality.
\begin{definition} An additive category with duality and weak equivalences is a
tuple $(\cat{C},T,\eta,w\cat{C})$ such that:
\begin{itemize}
\item $T$ is additive and $T$ and $\eta$ give a duality on $\cat{C}$
\item $T$ sends (opposites of) weak equivalences to weak equivalences
\item $(\cat{C},w\cat{C})$ is an additive category with weak equivalences
\end{itemize}
 \end{definition}
\begin{example}
  Let $(R,\alpha,\varepsilon)$ be a Wall-anti-structure. Then the category $P(R,\alpha,\varepsilon)$ becomes an additive category with duality and weak equivalences if we take the weak equivalences to be the isomorphisms.
\end{example}
To get a strict duality we can apply the functor $\cat{D}$ and because the functor $I \colon \cat{C} \to \cat{D}\cat{C}$ on underlying categories is an equivalence the category $\cat{D}\cat{C}$ will also be additive. Taking the weak equivalences in $\cat{D}\cat{C}$ to be pairs of maps in $w\cat{C}$ gives $\cat{D}(\cat{C},T,\eta)$ the structure of an additive category with duality and weak equivalences which is a functorial and better behaved replacement of  $(\cat{C},T,\eta,w\cat{C})$. There is a square of H-spaces and H-maps
\[\xymatrix{|Nw\cat{C}| \ar[d] \ar[r]^-{\eta_\cat{C}} &  \Omega|Nw\cat{C}(S^1)| \ar[d]\\
|Nw\cat{D}\cat{C}|  \ar[r]^-{\eta_{\cat{D}\cat{C}}} &  \Omega|Nw\cat{D}\cat{C}(S^1)|}\]
where the vertical maps are weak equivalences. Note that lemma \ref{lemma:strictify} also applies to additive categories with duality and weak equivalences, so that we may assume that our categories have astrict duality $T$, a duality preserving direct sum functor $(-\oplus-)$ which is strictly associative and strictly unital and that the unit $0$ is fixed under the duality(again, see \cite[5.11]{iblars}). 

\begin{remark}\label{remark:hyperbolic}Let $\xi \colon T(-)\oplus T(-) \to T(-\oplus-)$ be the canonical natural transformation. The category $Sym(w\cat{C})$ has a (functorial) sum operation $\perp$ called the orthogonal sum given by
\[(f \colon c \to T(c)) \perp (g \colon d \to Td) = c \oplus d \stackrel{f \oplus g}{\too} T(c)\oplus T(d) \stackrel{\xi_{c,d}}{\too} T(c \oplus d).\]
Under the induced operation the set $\pi_0(Sym(w\cat{C}))$ becomes a commutative monoid with unit represented by the $0$-form $0 \to 0$. For any object $c \stackrel{f}{\too} d$ of $Sd (w\cat{C})$ we can form the \emph{hyperbolic form} $H(f)$ on $c \stackrel{f}{\too} d$ which is the object 
\[\xymatrix{c\oplus T(d) \ar[rrr]^-{\begin{pmatrix}
        0 & T(f) \\ \eta_d \circ f & 0
    \end{pmatrix}} & & & T(c) \oplus TT(d) \ar[r]^-{\varphi_{c,T(d)}} & T(c \oplus T(d))}\]
of $Sym(w\cat{C})$. This is also compatible with maps in $Sdw\cat{C}$. Together the functors $\perp$ and $H$ give an ``action'' of $Sdw\cat{C}$ on $Sym(w\cat{C})$ analogous to the action of $M$ on $M^\g$ in section \ref{section:monoids}.
\end{remark}

Let $X$ be a pointed $\g$-set with $\sigma \colon X \to X$ representing the action of the non-trivial group element. The category $Q(X)$ inherits a strict duality $t$ by taking $t(U) = \sigma_\ast(U)$ and similarly for morphisms. If $\cat{C}$ is an additive category with weak equivalences and strict duality there is an induced duality $T_X$ on $w\cat{C}(X)$ given by taking a diagram 
\[A \colon Q(X)^{op} \to \cat{C}\]
to the composite diagram
\[Q(X)^{op} \stackrel{t}{\too} Q(X) \stackrel{A^{op}}{\too} \cat{C}^{op} \stackrel{T}{\too} \cat{C}.\]
Clearly the duality $T_X$ is strict and functorial in both $X$ and $(\cat{C},T,id)$. Let $n_+ = \{0,1, \ldots,n\}$ based at $0$ with the action of $\g$ taking an element $0 \neq k \geq 1 $ to $n-k$ and fixing $0$. If $X = 2_+ $ then the action interchanges the two non-trivial elements and the duality on $\cat{C}(2_+)$ sends the diagram
\[\xymatrix{X \ar@<-2pt>[r]_{i_1} & Y \ar@< 2pt>[r]^{p_2} \ar@<-2pt>[l]_{p_1} & X' \ar@< 2pt>[l]^{i_2}}\]
to the diagram
\[\xymatrix{TX' \ar@<-2pt>[r]_{Tp_2} & TY \ar@< 2pt>[r]^{Ti_1} \ar@<-2pt>[l]_{Ti_2} & TX \ar@< 2pt>[l]^{Tp_1}.}\]
We will always give $w\cat{C}^{\times^n}$ the strict duality given on objects by 
\[(X_1,\ldots,X_n) \mapsto (TX_n, \ldots , TX_1)\]
and similarly for maps. The evaluation map
\[e_n \colon w\cat{C}(n_+) \to w\cat{C}^{\times^n} \]
is equivariant for these dualities and is an equivalences of categories with duality (see \cite[4.11]{iblars}). 

We will now use the real simplicial set $S^{1,1} = \Delta R^1 / \partial \Delta R^1 $ to describe an action of $\g$ on the algebraic $K$-theory space of an additive category with strict duality and weak equivalences. For each $m \geq 0$ and each non-basepoint simplex $x \in S^{1,1}_{m}$ there is a unique $m$-simplex $\xi \in \Delta^1_m$ mapping to $x$ under the quotient map. The simplices $\Delta^1_m$ are linearly ordered by 
\[\xi \leq \zeta \iff |\xi^{-1}(\{1\})| \leq |\zeta^{-1}(\{1\})|,\]
and this gives a linear ordering of $S^{1,1}_{m}\setminus \{\ast\}$ which is reversed by the real simplicial structure map $w_m$. For each $n\geq 0$ the category $w\cat{C}(S^{1,1}_{n})$ inherits a duality $T_n$ from the action of $w_n$.  There are induced maps
\[w_{m,n} \colon N_nw\cat{C}(S^{1,1}_m) \to N_nw\cat{C}(S^{1,1}_m)\]
given by 
\[w_{m,n}(A_0 \stackrel{f_1}{\too} \ldots \stackrel{f_n}{\too} A_n) = (T_mA_n \stackrel{T_mf_n}{\too} \ldots \stackrel{ T_mf_1}{\too} T_mA_0),  \]
which satisfy the relations $ w_{m,n} \circ w_{m,n} = id$ and $w_{m,n} \circ (\alpha,\beta)^\ast = (\alpha^{op},\beta^{op})^\ast\circ w_{p,q}$ for maps $(\alpha,\beta) \colon ([m],[n]) \to ([p],[q])$ in $\Delta \times \Delta$. These assemble to a map of bisimplicial sets 
\[W \colon SdNw\cat{C}(SdS^{1,1}) \to SdNw\cat{C}(SdS^{1,1})\]
which in level $(m,n)$ is the map $Nw_{2m+1,2n+1}$. 

 The bisimplicial set $SdNw\cat{C}(SdS^{1,1})$ is naturally isomorphic to $NSdw\cat{C}(SdS^{1,1})$ and under this identification the map $W$ comes from a map of simplicial categories 
\[\tilde{W} \colon Sdw\cat{C}(SdS^{1,1}) \to Sdw\cat{C}(SdS^{1,1})\]
which squares to the identity and hence defines an action of $\g$ on $Sdw\cat{C}(SdS^{1,1})$. Let 
\[e_n^{1,1} \colon  w\cat{C}(S^{1,1}_n) \to w\cat{C}^{\times^{n}}\]
be the evaluation map which preserves the ordering of the underlying indexing set.  By \cite[4.11]{iblars} it is an equivalence of categories with duality and it induces a functor 
\[Sde^{1,1}_{2n+1}\colon Sdw\cat{C}(S^{1,1}_{2n+1}) \to Sdw\cat{C}^{\times^{2n+1}}  \]
which is $\g$-equivariant. The category $Sdw\cat{C}^{\times^{2n+1}}$ has the action given by 
\[(f_1, \ldots ,  f_{2n+1}) \mapsto (Tf_{2n+1}, \ldots ,Tf_{1}),\]
so a fixed object is of the form $(f_1, \ldots , f_n, f_{n+1},Tf_{n}, \ldots  ,Tf_{1})$ with $Tf_{n+1} = f_{n+1}$. We see that the last $n$ factors are redundant, so evaluation followed by projection on the first $n+1$ coordinates defines a functor
\[Sym (w\cat{C}(S^{1,1}_{2n+1})) \to Sdw\cat{C}^{\times^n} \times Sym (w\cat{C})\]
which is an equivalence of categories.

The map $d_0 \colon PS^1 \to S^1$ induces a map $Sdd_0 \colon SdPS^1 \to SdS^1$ and hence a map of simplicial categories
\[Sdw\cat{C}(SdPS^1) \to Sdw\cat{C}(SdS^1).\]
Define $Pb(\cat{C},T,w\cat{C})$ to be the pullback in the diagram
\[\xymatrix{Pb(\cat{C},T,w\cat{C}) \ar[r] \ar[d] & Sdw\cat{C}(SdPS^1) \ar[d] \\
Sym(w\cat{C}(SdS^{1,1})) \ar[r] & Sdw\cat{C}(SdS^1)}\]
of simplicial categories (without $\g$-actions) where the bottom map is the inclusion functor and the right hand vertical map is induced by $Sdd_0$. Note that the evaluation map gives an equivalence of categories (without duality)
\[Pb(\cat{C},T,w\cat{C})_n \simeq Sdw\cat{C}\times Sdw\cat{C}^{\times^n}\times Sym(w\cat{C}).\]
 Thinking of $Sym(w\cat{C})$ as a constant simplicial category we define a map of simplicial categories 
\[i \colon Sym(w\cat{C}) \to Pb(\cat{C},T,w\cat{C})\]
which in level $n$ sends an object $f \colon a \to Ta$ to the sum-diagram with value $f \colon  a \to Ta$ on subsets containing the unique non-trivial fixed point of $S^{1,1}_{2n+1}$ and $id \colon 0 \to 0 $ on subsets not containing it. The morphisms in $i_n(f)$ are identities or $0$ as for $\tilde{t}_n$. 
\begin{lemma}\label{lemma:catieq}
  The map $i$ induces a homology equivalence on nerves.
\end{lemma}
\begin{proof}Under the equivalences $Pb(\cat{C},T,w\cat{C}) \simeq Sdw\cat{C}\times Sdw\cat{C}^{\times^n}\times Sym(w\cat{C})$ the functor $i_n$ corresponds to the inclusion of $Sym(w\cat{C})$ by 
\[(f \colon a \to Ta) \mapsto (id \colon 0 \to 0, id \colon 0 \to 0 , \ldots, id \colon 0 \to 0 ,f \colon a \to Ta).\]
Let $k$ be a field and let $h_\ast(-)$ be homology with coefficients in $k$. We take $R$ to be the graded ring $h_\ast(NSdw\cat{C})$ and $P$ to be the graded left $R$-module $h_\ast(NSym(w\cat{C}))$, where the action comes from the one sketched in \ref{remark:hyperbolic}. The simplicial graded $k$-vector space $[n] \mapsto h_\ast(Pb(\cat{C},T,w\cat{C})_n)$ is isomorphic to the bar construction $B(R,R,P)$ and the map in homology induced by $i$ is the inclusion $P \inj B(R,R,P)$ given on generators by 
\[p \mapsto 1 \otimes 1 \otimes \cdots \otimes 1 \otimes p.\]
This is a quasi-isomorphism of simplicial graded $k$-vector spaces. To see that the map preserves the $R$-module structure after taking homology we observe that there is a retraction back onto $P$ given like the map $r$ in \ref{lemma:ji}. Using the spectral sequence
\[E^{p,q}_2 = H_p(h_q(X)) \implies h_{p+q}(dX)\] for bisimplicial sets $X$ (see e.g., \cite[IV,2]{gj}) we get an isomorphism on homology with $k$ coefficients and, since this holds for any field $k$, an isomorphism on homology with integral coefficients. 
\end{proof}

Now assume that $\cat{C}$ has an object $t$ whose class in $\pi_0Nw\cat{C}$ is a cofinal generator. The subdivision of the functor $t\oplus- \colon w\cat{C} \to w\cat{C}$ is the functor that adds $t \stackrel{id}{\too} t$ to objects $a \to b$ of $Sdw\cat{C}$. Similarly, the subdivision $Sd\tilde{t}_n\oplus-$ of the functor $\tilde{t}\oplus-$, defined earlier, adds the map of sum-diagrams $\tilde{t}_n \stackrel{id}{\too} \tilde{t}_n$ to objects $A \to B$ in $Sdw\cat{C}(S^{1,1}_{n})$. For each $n\geq 0$ there is a diagram
\[ Sdw\cat{C}(S^1_{n})  \stackrel{Sd\tilde{t}_{n} \oplus -}{\too}  Sdw\cat{C}(S^1_{n})   \stackrel{Sd\tilde{t}_{n} \oplus -}{\too} \cdots \]
and we define $Sdw\cat{C}_\infty$ and $Sdw\cat{C}(S^1_{n})_\infty$ to be the Grothendieck constructions on the diagrams as before. Also the map
\[(Sdd_0)_\ast \colon  Sdw\cat{C}(S^1_{2n+2}) \to Sdw\cat{C}(S^1_{2n+1})\]
commutes with the maps $Sd\tilde{t}_n\oplus-$ and just as before there is an induced map
\[Sdw\cat{C}(SdPS^1)_\infty \to Sdw\cat{C}(SdS^1)\]
which induces a homology fibration on nerves. The maps $Sd\tilde{t}_n\oplus-$ also induce a map on the pullback $Pb(\cat{C},T,w\cat{C})$ which commutes with the projection to $Sym(w\cat{C}(SdS^{1,1}))$. There results a pullback square of simplicial categories 
\[\xymatrix{Pb(\cat{C},T,w\cat{C})_\infty \ar[r] \ar[d] & Sdw\cat{C}(SdPS^1)_\infty \ar[d] \\
Sym(w\cat{C}(SdS^{1,1})) \ar[r] & Sdw\cat{C}(SdS^1)}\]
where the vertical maps induce homology fibrations on nerves. The inclusion of $ Pb(\cat{C},T,w\cat{C})$ into $Pb(\cat{C},T,w\cat{C})_\infty$ at the start of the diagram defining the latter will be called $j$.

\begin{lemma}The map
\[j \circ i \colon Sym(w\cat{C}) \to Pb(\cat{C},T,w\cat{C})_\infty\]
induces an isomorphism
\[H_\ast(NSym(w\cat{C})[\pi_0Nw\cat{C}^{-1}] \iso H_\ast(NPb(\cat{C},T,w\cat{C})_\infty) \]
of left $H_\ast(NSdw\cat{C})$-modules.
\end{lemma}
\begin{proof} By lemma \ref{lemma:catieq} the map $i$ induces an isomorphism $H_\ast(NSym(w\cat{C})) \cong H_\ast(NPb(\cat{C},T,w\cat{C}))$ of left $H_\ast(NSdw\cat{C})$-modules. The map 
\[Sd\tilde{t} \colon Pb(\cat{C},T,w\cat{C}) \to Pb(\cat{C},T,w\cat{C})\]
induces left multiplication by $[t]$ on $H_\ast(NPb(\cat{C},T,w\cat{C}))$, and by Thomason's theorem we get a sequence of isomorphisms 
\begin{align*} H_\ast(NPb(\cat{C},T,w\cat{C})_\infty) &\cong \colim \left(H_\ast \left( NSym(w\cat{C})\right) \stackrel{[t]\cdot}{\too} H_\ast\left(NSym(w\cat{C})\right) \stackrel{[t]\cdot}{\too}   \ldots\right) \\
&\cong H_\ast(NSym(w\cat{C}))[t^{-1}]\\
&\cong H_\ast(NSym(w\cat{C}))[\pi_0Nw\cat{C}^{-1}]
\end{align*}
of left $H_\ast(NSdw\cat{C})$-modules as desired.

\end{proof}
 The proof of the following statement is similar to that of theorem \ref{theorem:mainmon}. We use that there is a natural ring isomorphism $H_\ast(NSdw\cat{C}) \cong H_\ast(Nw\cat{C})$.
\begin{theorem}\label{theorem:maincat}
  Let $(\cat{C},w\cat{C},T)$ be an additive category with strict duality and weak equivalences. Then the map $|NSym(w\cat{C})| \to (\Omega^{1,1}|NSymw\cat{C}(S^{1,1})|)^\g$ induces isomorphisms 
\[\pi_0(NSymw\cat{C})[\pi_0Nw\cat{C}^{-1}] \iso \pi_0((\Omega^{1,1}|NSymw\cat{C}(S^{1,1})|)^\g)\]
of monoids and 
\[H_\ast(NSymw\cat{C})[\pi_0Nw\cat{C}^{-1}] \to H_\ast((\Omega^{1,1}|NSymw\cat{C}(S^{1,1}))^\g)\]
of left $H_\ast(Nw\cat{C})$-modules.
\end{theorem}

For a Wall-anti-structure $(R,\alpha,\varepsilon)$ we set
\[K^{1,1}(R,\alpha,\varepsilon) = (\Omega^{1,1}|NSym (i\cat{D}P(R,\alpha,\varepsilon))(S^{1,1})|)^\g\]
and 
\[K^{1,1}_n(R,\alpha,\varepsilon) = \pi_n K^{1,1}(R,\alpha,\varepsilon).\]
We will now investigate the two fundamental cases when $R = \Z$ and $\alpha = id_\Z$. They are $\varepsilon = 1$ and $\varepsilon  = -1$. In the first case observe that $Sym(iP(\Z,id_\Z,1)$ is the category of non-degenerate symmetric bilinear form spaces over $\Z$.
\begin{proposition}\label{prop:symm(Z)}
  The monoid $K^{1,1}_0(\Z,id_\Z,1)$ is not a group.
\end{proposition}
\begin{proof}
  By theorem \ref{theorem:maincat} there is an isomorphism 
\[K^{1,1}_0(\Z,id_\Z,1) \cong \pi_0|NSym (i\cat{D}P(\Z,id_\Z,1))|[\pi_0(Ni\cat{D}(P(\Z)))^{-1}]\]
and the right hand side is isomorphic to the monoid 
\[M = \pi_0NSym(iP(\Z,id_\Z,1))[\pi_0(Ni(P(\Z)))^{-1}].\]
We will show that the latter is not a group by finding an element that cannot have an inverse. 

The $n$-th hyperbolic space $H^n$ is the symmetric bilinear form space with underlying abelian group $\Z^{2n}$ and the symmetric form given by the matrix 
\[\begin{pmatrix}0 & I_n\\ I_n & 0 \end{pmatrix},\]
where $I_n$ denotes the $n \times n$ identity matrix. The monoid $\pi_0(Ni(P(\Z)))$, which is isomorphic to $\N$, acts on $\pi_0NSym(iP(\Z,id_\Z,1))$ by adding hyperbolic spaces $H^n$ via the orthogonal sum. Let $\langle 1 \rangle$ denote the object $\Z \iso Hom(\Z,\Z)$ which sends an integer $k$ to the multiplication by $k$ map. Assume that $[\langle 1 \rangle]$ has an inverse in $M$. Elements of $M$ can be represented as differences $[a] - [H^m]$ where $a$ is in $Sym(iP(\Z,id_\Z,1)$. An inverse for $[\langle 1 \rangle]$ is a difference $[a] - [H^m]$ such that $[\langle 1 \rangle] + [a] - [H^m] = 0$ in $M$, or equivalently such that for some $n$ the equation
\[\langle 1 \rangle] + [a] + [H^n] = [H^m] + [H^n]\]
holds in $\pi_0NSym(iP(\Z,id_\Z,1))$. Since $H^m \perp H^n \cong H^{m+n}$, this means that we have an isomorphism
\[\langle 1 \rangle \perp a \perp H^n \cong H^{m+n}.\]
On the left hand side the element $(1,0,0)$ pairs with itself to $1 \in \Z$ under the bilinear form. However, on the right hand side an element $x$ in $H^{m+n}$ is of the form
\[
x = \begin{pmatrix}  x_1 \\ \vdots \\ x_{m+n} \\ x'_1 \\ \vdots \\ x'_{m+n} \end{pmatrix} \in \Z^{2(m+n)}
\]
and its pairing with itself is the matrix product
\[ x^T \begin{pmatrix}0 & I_n\\ I_n & 0 \end{pmatrix} x = \sum_{i=1}^{m+n}x_ix'_i + \sum_{i=1}^{m+n}x'_ix_i = 2(\sum_{i=1}^{m+n}x_ix'_i).\]
Since $1$ is an odd number we conclude that such an isomorphism cannot exist and that $K^{1,1}_0(\Z,id_\Z,1)$ is not a group.
\end{proof}

The second case is $Sym(iP(\Z,id_\Z,-1))$, the category of non-degenerate symplectic bilinear form spaces over $\Z$. We write ${}_{-1}H^n(\Z)$ for the symplectic form module with matrix 
\[\begin{pmatrix}0 & I_n\\ -I_n & 0 \end{pmatrix}.\]
By \cite[4,3.5]{milhus} any symplectic form module over $\Z$ is isomorphic to ${}_{-1}H^n(\Z)$ for a uniquely determined $n\geq 0$. We call this number the rank of the symplectic module. The corresponding rank map 
\[\pi_0|NSym(iP(\Z,id_\Z,-1)| \to \N\]
is an isomorphism of monoids.
\begin{proposition}\label{prop:symp(Z)}
  The rank map induces an isomorphism 
\[K^{1,1}_0(\Z,id_\Z,1) \cong \Z .\]
\end{proposition}

\appendix
\section{Appendix: The category $\Delta R$}\label{appendix}

The category $\Delta R$ has the same objects as the finite ordinal category $\Delta$ but more morphisms. In addition to the maps of $\Delta$ there is for each $n\geq 0$ morphism $\omega_n \colon [n] \to [n]$, which should be thought of as reversing the ordering on $[n]$. The maps satisfy the relations
\begin{align}\label{rel:realsimp}
  \omega_n \circ \omega_n &= id_{[n]} \\
\omega_n \circ \sigma^j &= \sigma^{n-j} \circ \omega_{n+1}\\
\omega_n \circ \delta^i &= \delta^{n-i} \circ \omega_{n-1}
\end{align}
for $0 \geq i,j \geq n$. Following \cite{iblars} a functor from $(\Delta R)^{op}$ to sets is called a real simplicial set and similarly for functors into other categories. The maps induced in a real simplicial object by $\omega_n$ is denoted by $w_n$. 

If we restrict a real simplicial set $X$ to $\Delta^{op}$ the realization $|(X|_{\Delta^{op}})|$carries an action of $\g$ which for $(x,t_0,\ldots,t_n) \in X_n \times \Delta^n$ acts by
\[[(x,t_0,\ldots,t_n)] \mapsto [(w_n(x),t_n, \ldots ,t_0)],\]
(see \cite[1]{iblars} for details). Recall the functor
\[(-)^{op} \colon \Delta \to \Delta\]
which is the identity on objects and which sends $\delta^i \colon [n] \to [n+1]$ to $(\delta^i)^{op} = \delta^{n+1 -i}$ and $\sigma^j \colon [n] \to [n-1]$ to $(\sigma^j)^{op} = \sigma^{n-1-j}$. Clearly $(-)^{op}\circ(-)^{op} = Id_\Delta$. Let $\cat{A}$ be any category. Given a simplicial object $X \colon \Delta^{op} \to \cat{A}$ its opposite is defined by 
\[X^{op} = X \circ (-)^{op}.\] 
This defines a functor on the functor category $\cat{A}^{\Delta^{op}}$ which squares to the identity.
\begin{lemma}\label{lemma:realstr}
  Giving an extension of a functor $X \colon \Delta^{op} \to \cat{A} $ to the category $\Delta R^{op}$ is equivalent to giving a map $\omega \colon X^{op} \to X$ such that $\omega \circ \omega^{op} = id_X$. 
\end{lemma}

Now recall the functor $Sd \colon \Delta \to \Delta$ given by $Sd[n] = [2n+1]$ and $Sd(\theta) = \theta \sqcup \theta^{op}$. The Segal edgewise subdivision of $X$ is defined by $SdX = X \circ Sd$. It gives an endofunctor on $\cat{A}^{\Delta^{op}}$ and it is not hard to see that $Sd \circ (-)^{op} = Sd$, so that $SdX^{op} = SdX$ for any simplicial object $X$. Given a real simplicial object $Y \colon (\Delta R)^{op} \to \cat{A}$ we can regard it as a simplicial object $Y|_{\Delta^{op}}$ with a map $\omega \colon Y^{op} \to Y$ as in lemma \ref{lemma:realstr}. Then on the subdivision we get a map
\[Sd(\omega) \colon SdX^{op} = SdX \to SdX\]
such that $Sd(\omega)^2 = id_{SdX}$. In other words, $Sd(\omega)$ defines an action of $\g$ on $SdX$. For a real simplicial set $X$ the natural homeomorphism $|Sd(X|_{\Delta^{op}})| \iso |X|_{\Delta^{op}}|$ of \cite[A.1]{confiter} is $\g$-equivariant.
\begin{example}The representable functor $\Delta R^1 = hom_{\Delta R^1}(-,[1])$ realizes to the topological 1-simplex with the action of $\g$ given by reflection through the middle point. Its boundary $\partial\Delta R^1$ realizes to the two end points, which are interchanged by the $\g$-action. Define $S^{1,1}$ to be $\Delta R^1/\partial\Delta R^1$. The realization $|S^{1,1}|$ is $\g$-homeomorphic to the one point compactification of the $-1$-representation of $\g$ on $\R$. 
\end{example}

\bibliographystyle{alpha}
\bibliography{eqloops}

\end{document}